\documentclass{amsart}
\usepackage{amsmath,amsxtra,amssymb,amsthm,latexsym}
\usepackage[mathscr]{eucal}
\usepackage{color,verbatim}
\allowdisplaybreaks
\theoremstyle{plain}
\newtheorem*{conjecture}{Conjecture}

\newtheorem{theorem}{Theorem}
\newtheorem{proposition}{Proposition}
\newtheorem{lemma}{Lemma}
\newtheorem{corollary}{Corollary}

\newtheorem{Definition}{Definition}
\newtheorem{Remark}{Remark}
\numberwithin{equation}{section}

\DeclareMathOperator{\gl}{GL}

\DeclareMathOperator{\so}{SO}

\DeclareMathOperator{\symp}{Sp}

\newcommand{\aaeq}{\underset{\text{\,a.a.}}{=}}

\newcommand{\mA}{\mathbb{A}}
\newcommand{\mC}{\mathbb{C}}

\makeatletter
\def\Ddots{\mathinner{\mkern1mu\raise\p@
\vbox{\kern7\p@\hbox{.}}\mkern2mu
\raise4\p@\hbox{.}\mkern2mu\raise7\p@\hbox{.}\mkern1mu}}
\makeatother
\title{Refined global Gross-Prasad conjecture
on special Bessel periods  and
B\"ocherer's conjecture
}
\author{Masaaki Furusawa}
\address[Masaaki Furusawa]{
Department of Mathematics, Graduate School of Science,
              Osaka City University,
         Sugimoto 3-3-138, Sumiyoshi-ku, Osaka 558-8585, Japan
}
\email[Masaaki Furusawa]{furusawa@sci.osaka-cu.ac.jp}
\thanks{The research of the first author was supported in part by 
JSPS KAKENHI Grant Numbers 
JP25400020, 
JP16K05069.}
\author{Kazuki Morimoto}
\address[Kazuki Morimoto]{
Department of Mathematics, Graduate School of Science, 
Kobe University, 1-1 Rokkodai-cho, Nada-ku, Kobe 657-8501, Japan
}
\email[Kazuki Morimoto]{morimoto@math.kobe-u.ac.jp}
\thanks{The research of the second author was supported in part by
JSPS KAKENHI Grant Number JP26800021.}
\date{\today}                                           
\dedicatory{To the memory of Joseph~Shalika}
\begin{document}
%
%
%
%
%
%
%
%
%
%
\begin{abstract}
In this paper we pursue the refined global Gross-Prasad conjecture
for  Bessel periods formulated by Yifeng Liu
in the case of special Bessel periods 
for $\so\left(2n+1\right)\times\so\left(2\right)$.
Recall that a Bessel period for $\so\left(2n+1\right)\times\so\left(2\right)$
is called \emph{special}
when the representation of $\so\left(2\right)$ is trivial.
Let $\pi$ be an irreducible cuspidal tempered automorphic 
representation of a special orthogonal group
of an odd dimensional quadratic space
over a totally real number field $F$
whose local component $\pi_v$
at any archimedean place $v$ of $F$ is
a discrete series representation.
Let $E$ be a quadratic extension of $F$ and 
suppose that the special Bessel period corresponding to $E$
does not vanish identically on $\pi$.
Then we prove the Ichino-Ikeda type explicit formula conjectured by Liu
for the central value  $L\left(1/2,\pi\right)L\left(1/2,\pi\times\chi_E\right)$,
where $\chi_E$ denotes the quadratic character corresponding to $E$.
Our result yields a proof of B\"ocherer's conjecture on
holomorphic Siegel cusp forms of degree two which are Hecke eigenforms.
\end{abstract}
%
%
%
\maketitle
%
%
%
%
%
%
\section{Introduction}
Research on special values of arithmetic $L$-functions
is one of the pivotal subjects in number theory.
The central values are of particular interest because of the Birch
and Swinnerton-Dyer conjecture and its natural generalizations.

In the early 1990s, Gross and Prasad~\cite{GP1,GP2} proclaimed 
a conjecture concerning a relationship between 
non-vanishing of certain period integrals on special orthogonal
groups and non-vanishing of central values of certain tensor product
$L$-functions, together with the local counterpart conjecture.
Recently Gan, Gross and Prasad~\cite{GGP} extended the conjecture
to classical groups and metaplectic groups.
On the other hand, Ichino and Ikeda, 
in their very influential paper~\cite{II},  
refined the Gross-Prasad conjecture and formulated a conjectural precise formula 
for the central $L$-value in terms of the period integral 
for tempered cuspidal automorphic representation in
the $\mathrm{SO}\left(n+1\right)\times\mathrm{SO}\left(n\right)$
case, i.e. co-dimension $1$ case.
Inspired by \cite{II}, Harris~\cite{H} formulated a similar conjectural formula
in the  co-dimension $1$ unitary group case.
Recently Liu~\cite{Liu} extended the work of Ichino-Ikeda
and Harris
to  Bessel periods for  orthogonal and unitary groups
and formulated a conjectural  precise formula expressing the central $L$-values
in terms of the Bessel periods in the arbitrary co-dimension case.
%

In our previous paper~\cite{FM0},
we investigated the Gross-Prasad conjecture
for the special Bessel periods on $\mathrm{SO}\left(2n+1\right)\times
\mathrm{SO}\left(2\right)$
and proved that the non-vanishing of the period implies
the non-vanishing of the corresponding central $L$-value.
In this paper, we refine the results in  \cite{FM0}
and prove the Ichino-Ikeda type precise $L$-value formula conjectured
by Liu~\cite{Liu} in the aforementioned case.
As a corollary, we also obtain a proof of the long-standing
conjecture by
B\"ocherer in \cite{Bo}, concerning
central critical values of imaginary quadratic twists
of spinor $L$-functions for holomorphic Siegel cusp forms
of degree two which are Hecke eigenforms, thanks to
the beautiful work by Dickson, Pitale, Saha and Schmidt~\cite{DPSS}.
%
%
%
%

In order to state our main results, 
let us introduce notation.
For the convenience of the reader, we shall 
use as much as possible 
the notation  in \cite{FM0}, to which
this paper is a sequel.
%
%
%
\subsection{Notation}\label{ss: notation}
%
Let $F$ be a number field and $\mathbb A_F$ its ring of adeles.
We shall often abbreviate $\mathbb A_F$ as $\mathbb A$
for simplicity.
Let $\psi$ be a non-trivial  character of $\mA$ which is trivial on $F$.
For $a\in F^\times$,  we denote
by $\psi^a$ the character of $\mA$ 
defined by $\psi^a(x) = \psi(ax)$.
For a place $v$ of $F$, let $F_v$ be the completion of $F$ at $v$
and $\psi_v$ the character of $F_v$ induced by $\psi$.
When $v$ is non-archimedean, we write by $\mathcal O_v$ and
$\varpi_v$,
the ring of integers in $F_v$ and
a prime element of $F_v$, respectively.

Let $E$  be a quadratic extension field 
of $F$ and $\chi_E$ 
 the quadratic character of $\mA_F^\times \slash F^\times$
corresponding to $E$.
Throughout the paper, we fix $E$. 
We simply write $\chi$ for $\chi_E$ when there is no fear of confusion.

For a positive integer $n\ge 2$, let $\mathcal G_n=\mathcal G_{n, E}$
denote a certain set of $F$-isomorphism classes of special
orthogonal groups defined as follows.
Let $\left(V, \left(\,\,,\,\,\right)\right)$ be a quadratic space
over $F$, i.e. a finite dimensional vector space over $F$ equipped with
a non-degenerate symmetric bilinear form $\left(\,\,,\,\,\right)$.
We suppose that $\dim V=2n+1$, the Witt index of $V$ is at least $n-1$
and $V$ has an orthogonal direct sum decomposition
$V=\mathbb H^{\,n-1}\oplus L$ 
where $\mathbb H$ denotes the hyperbolic plane
over $F$ and $L$ is a three dimensional quadratic 
space containing  $\left(E,c \cdot \mathrm{N}_{E\slash F}\right)$
for some $c \in F^\times$.
Then we define $\mathcal G_n$ as the set of $F$-isomorphism classes
of the special orthogonal groups $\mathrm{SO}\left(V\right)$
for such $V$.
Let $\operatorname{disc}\left(V\right)$ 
denote the discriminant of $\left(V,(\,\,,\,\,)\right)$
which takes a value in $F^\times\slash \left(F^\times\right)^2$.
We often denote the
quadratic space $\left(V, b\cdot
\left(\,\,,\,\,\right)\right)$ simply as $b\, V$.
We note that then $\operatorname{disc}\left(b\,V\right)=b\cdot\operatorname{disc}
\left(V\right)\in F^\times\slash \left(F^\times\right)^2$ and 
$\mathrm{SO}\left(b\,V\right)=\mathrm{SO}\left(V\right)$.
Thus from now on we shall assume 
$\operatorname{disc}\left(V\right)=\left(-1\right)^{n}$,
i.e. $\operatorname{disc}\left(L\right)=-1$,
without loss of generality.
We shall often identify the group $\mathrm{SO}\left(V\right)$ with
its isomorphism class in $\mathcal G_n$ by abuse of notation.
Let us denote by $\mathbb V=\mathbb V_n$ 
such a quadratic space with 
$\dim \mathbb{V} = 2n+1$ and the Witt index $n$, which is uniquely determined
up to a scalar multiplication, and
we write  its special orthogonal group
$\mathrm{SO}\left(\mathbb V\right)$
(and its $F$-isomorphism class)
by $\mathbb G=\mathbb G_n$. 
We note that  $\mathbb G$ splits over $F$.

Throughout the paper, for an algebraic group $\mathbf G$ 
defined over $F$, we write $\mathbf G_v$ for $\mathbf G\left(F_v\right)$
and
we always take the measure $dg$ on $\mathbf G\left(\mA\right)$ to
be the Tamagawa measure, 
unless specified otherwise.
For each $v$, we take the self-dual measure with respect to $\psi_v$
on $F_v$.
Then recall that the product measure on $\mA$ is the self-dual measure
with respect to $\psi$ and  is also the Tamagawa measure since $\mathrm{Vol}\left(
\mA\slash F\right)=1$.
For a unipotent algebraic group $\mathbf U$ defined over $F$,
we also specify the local measure $du_v$ on $\mathbf{U}_v$
to be the measure corresponding to the gauge form defined over $F$, 
together with our choice of the measure on $F_v$,
at each place $v$ of $F$.
Then for $du=\prod_v du_v$, we have
$\mathrm{Vol}\left(\mathbf U\left(F\right)\backslash
\mathbf U\left(\mA\right), du\right)=1$
and 
$du$ is the Tamagawa measure on $\mathbf U\left(\mathbb A\right)$.
%
%
%
\subsection{Special Bessel periods}
\label{ss:Special Bessel periods}
Let $G=\mathrm{SO}\left(V\right)\in\mathcal G$.
First we decompose $V$ as a direct sum
$
V=X^+\oplus L\oplus X^-
$
where $X^\pm$ are totally isotropic $\left(n-1\right)$-dimensional
subspaces of $V$ which are dual to each other and orthogonal to $L$.
When $G=\mathbb G$, i.e. $V=\mathbb V$, we extend $X^+$ to $V^+$
and $X^-$ to $V^-$ respectively so that $V^\pm$ are totally isotropic
$n$-dimensional subspaces of $\mathbb V$ which are dual to each other.
We take a basis $\left\{e_1,\cdots, e_{n-1}\right\}$ of $X^+$
and a basis $\left\{e_{-1},\cdots , e_{-n+1}\right\}$ of $X^-$
respectively so that
\begin{equation}\label{e: kronecker}
\left(e_i,e_{-j}\right)=\delta_{i,j}
\quad
\text{for $1\le i,j\le n-1$},
\end{equation}
where $\delta_{i,j}$ denotes Kronecker's delta.
When $V=\mathbb V$, we take $e_n\in V^+$ and
$e_{-n}\in V^-$ respectively so that \eqref{e: kronecker} holds
for $1\le i,j\le n$.
We also  fix a basis of $L$.
When $V=\mathbb V$, 
we take it  to be of the form
$\left\{e_{-n},e,e_{n}\right\}$
where $e$ is a vector in $L$ orthogonal to
$e_{-n}$ and $e_n$ with $\left(e,e\right)=1$.
Then for a matrix representation of elements of $G$, 
as a basis of $V$, we employ
\[
e_{-1},\cdots , e_{-n+1},\,
\text{basis of $L$}, \,e_{n-1},\cdots, e_1.
\]

We denote by  $P^\prime$ the maximal parabolic subgroup of $G$
defined as the stabilizer of
the isotropic subspace $X^-$.
Let 
\begin{equation}\label{e: P prime}
P^\prime=M^\prime S^\prime
\end{equation}
be the Levi decomposition where $M^\prime$ and $S^\prime$ denote the Levi part
and the unipotent part of $P^\prime$ respectively.
Let us take $\lambda\in F^\times$ so that 
$E=F\left(\sqrt{\lambda}\,\right)$.
Since $L$ contains the quadratic space $\left(E,c\cdot \mathrm{N}_{E\slash F}\right)$
and $\mathrm{disc}\left(L\right)=-1$,
we may take $e_\lambda\in L\left(F\right)$ such that 
$\left(e_\lambda,e_\lambda\right)=\lambda$ and 
we fix it once and for all.
Then there is a homomorphism from $S^\prime$ to $\mathbb G_a$ 
defined by
\[
\begin{pmatrix} 1_{n-1}&A&B\\ 0&1_3&A^\prime\\0 &0&1_{n-1}\end{pmatrix}
\mapsto \left(Ae_\lambda, e_{n-1}\right),
\]
where we regard $A$ as an element of $\mathrm{Hom}\left(L,X^-\right)$
and $\left(\,\, ,\,\,\right)$ is the symmetric bilinear form on $V$,
and its stabilizer in the Levi component $M^\prime$
is given by
\[
\left\{ \begin{pmatrix} p&0&0\\ 0&h&0\\ 0&0&p^\ast \end{pmatrix} : p \in \mathcal{P}_{n-1}, \, h \in \mathrm{SO}(L), h e_\lambda = e_\lambda \right\}
\]
where $\mathcal{P}_{n-1}$ denotes the mirabolic subgroup of $\mathrm{GL}_{n-1}$, i.e.
\[
\mathcal{P}_{n-1} = \left\{ \begin{pmatrix}\alpha &u\\0 &1 \end{pmatrix} : \alpha \in \mathrm{GL}_{n-2}, u \in \mathbb{G}_a^{n-2} \right\},
\]
and $p^\ast=J_{n-1}{}^tp^{-1}J_{n-1}$.
Here $J_r$ denotes the $r\times r$ matrix with ones on the sinister diagonal, zeros
elsewhere.
Let $U_{n-1}$ denote the group of upper unipotent matrices in
$\mathrm{GL}_{n-1}$.
We define $\check{u}\in M^\prime$ 
for $u\in U_{n-1}$ by
\begin{equation}\label{e: u check}
\check{u}=\begin{pmatrix}u&0&0\\0&1_3&0\\0&0&u^\ast\end{pmatrix}
\end{equation}
and let $S$ be a unipotent subgroup of $P^\prime$ defined by
\begin{equation}\label{e:s}
S:=S^\prime S^{\prime\prime}\quad
\text{where $S^{\prime\prime}=
\left\{\check{u}: u\in U_{n-1}\right\}$}.
\end{equation}
Let us define a subgroup $D_{\lambda}$ of $M^\prime$ by
\[
D_{\lambda} := \left\{ 
\begin{pmatrix} 1_{n-1}&0&0\\ 0&h&0\\ 0&0&1_{n-1}\end{pmatrix} : h \in \mathrm{SO}(L), \,h e_\lambda = e_\lambda \right\}
\simeq \mathrm{SO}\left(E\right)
\simeq  E^\times \slash F^\times.
\]
\begin{Definition}\label{d: Bessel}
The Bessel subgroup $R_\lambda$ of $G$ is defined by
\[
R_\lambda:=D_\lambda \,S
\]
and we define a character $\chi_\lambda$ of $R_\lambda\left(\mathbb A\right)$
by setting
$\chi_\lambda\left(t\right):=1$ for $t\in D_\lambda\left(\mathbb A\right)$
and
\begin{equation}\label{e: character identity}
\chi_\lambda\left(s^\prime \check{u}\right)
=\psi\left(\left(Ae_\lambda, e_{n-1}\right)\right)
\,\psi\left(u_{1,2}+\cdots+u_{n-2,n-1}\right)
\end{equation}
for  
\[
s^\prime=
\begin{pmatrix} 1_{n-1}&A&B\\ 0&1_3&A^\prime\\ 0&0&1_{n-1}\end{pmatrix}
\in S^\prime\left(\mathbb A\right)\quad\text{and}\quad
u=\left(u_{i,j}\right)\in U_{n-1}\left(\mathbb A\right).
\]

Then for an automorphic form $\phi$ on $G\left(\mathbb A\right)$,
its special Bessel period of type $E$ is defined by
\[
B_{\lambda, \psi}\left(\phi\right)
=\int_{R_\lambda\left(F\right)\backslash R_\lambda\left(\mathbb A\right)}
\phi\left(r\right)\chi_\lambda\left(r\right)^{-1}\, dr.
\]
\end{Definition}
We refer to \cite[(5)]{FM0} for the dependency of
$B_{\lambda, \psi}$ on the choice of $\lambda$ and $e_\lambda$.
%
%
%
%
%
%
%
%
\subsection{Refined Gan-Gross-Prasad conjecture by Liu
in our case}
\label{Refined Gan-Gross-Prasad conjecture by Liu
in our case}
%
Let $\pi$ be an irreducible \emph{tempered} cuspidal automorphic representation
of $G\left(\mathbb A\right)$ for $G\in\mathcal G$
and $V_\pi$ its space of automorphic forms.

Let $\langle\, \,,\,\,\rangle$ denote the $G\left(\mathbb A\right)$-invariant
Hermitian inner product on $V_\pi$
given by the Petersson inner product, i.e.
\[
\langle\phi_1,\phi_2\rangle=
\int_{G\left(F\right)\backslash G\left(\mathbb A\right)}
\phi_1\left(g\right)\overline{\phi_2\left(g\right)}\,
dg
\quad\text{for $\phi_1,\phi_2\in V_\pi$.}
\]
Since $\pi=\otimes_v \,\pi_v$ where $\pi_v$ is unitary,
we may also choose a $G_v$-invariant Hermitian
inner product $\langle\,\,,\,\,\rangle_v$ on the space $V_{\pi_v}$ of $\pi_v$
  for each place $v$
so that 
 \[
 \langle\phi_1 ,\phi_2\rangle=\prod_v\langle\phi_{1,v} ,\phi_{2,v}\rangle_v
 \]
for any decomposable vectors $\phi_1=\otimes_v\,\phi_{1,v}$ and 
$\phi_2=\otimes_v\,\phi_{2,v}\in V_\pi$.

We choose a local Haar measure 
$dg_v$ on $G_v$ for each place $v$ of $F$
so that $\mathrm{Vol}\left(K_{G,v},dg_v\right)=1$
at almost all  $v$, where
$K_{G,v}$ is a maximal compact subgroup of $G_v$.
Let us also choose a local Haar measure $dt_v$
on $D_{\lambda,v}=D_\lambda\left(F_v\right)$ at each place $v$ of $F$
so that $\mathrm{Vol}\left(K_{\lambda,v},dt_v\right)=1$
at almost all $v$, where $K_{\lambda,v}$
is a maximal compact subgroup of $D_{\lambda,v}$.
We define positive constants $C_G$ and $C_\lambda$,
called Haar measure constants in \cite{II},  by
%
%
%
\begin{equation}\label{e: measure comparison}
dg=C_G\cdot\prod_v \,dg_v\quad\text{and}\quad
dt=C_\lambda\cdot \prod_v\, dt_v
\end{equation}
respectively.
Here we recall that $dg$ and $dt$ are the Tamagawa measures
on $G\left(\mA\right)$ and $D_\lambda\left(\mA\right)$,
respectively.
%
%
\subsubsection{Local integral}
At each place $v$ of $F$, a local integral
$\alpha_v\left(\phi_v,\phi_v^\prime\right)$ for 
$\phi_v,\,\phi_v^\prime\in V_{\pi_v}$ is defined as follows.
%

First suppose that $v$ is non-archimedean.
Let us first recall the definition of a \emph{stable integral} by 
Lapid and Mao~\cite[Definition~2.1,Remark~2.2]{LM0}. 
%
\begin{Definition}\label{d: stable integral}
Let $U$ be a unipotent group over $F_v$ and $f$ a locally constant function on $U$.
We say that \emph{$f$ has a stable integral over $U$} if 
there exists a compact open subgroup $N$ of $U$ such that for 
any compact open subgroup $N^\prime$ of $U$ containing  $N$ we have 
\[
\int_{N^\prime} f(u) \, du = \int_{N} f(u) \, du.
\]
Then
we denote this common value by $\displaystyle{\int_{U}^{\mathrm{st}} f(u) \, du}$
and say that the integral  stabilizes at $N$.
\end{Definition}
%
\begin{Remark}\label{r: stable integral}
Note that if $f\in L^1\left(U\right)$ and $f$ has a stable integral over $U$,
then we have 
\[
\displaystyle{\int_{U} f(u) \, du=\int_U^{\mathrm{st}}f\left(u\right)\, du}.
\]
\end{Remark}
%
\begin{Definition}\label{alpha}
For a non-archimedean place  $v$, 
we  define $\alpha_v\left(\phi_v,\phi^\prime_v\right)$ for $\phi_v,\phi_v^\prime\in V_{\pi_v}$
by
\begin{equation}\label{e: local integral 1}
\alpha_v\left(\phi_v,\phi_v^\prime\right)
:=
\int_{D_{\lambda, v}}\int_{S_v}^{\mathrm{st}}
\langle\pi_v\left(s_vt_v\right)\phi_v,\phi_v^\prime\rangle_v\,\,
\chi_\lambda\left(s_v\right)^{-1}\, ds_v\,dt_v.
\end{equation}
\end{Definition}
Indeed it is shown in Liu~\cite{Liu} that for any $t_v\in D_{\lambda, v}$
the inner integral of \eqref{e: local integral 1} stabilizes at a certain open compact subgroup of $S_v$
 \cite[Proposition~3.1]{Liu} and 
the outer integral of \eqref{e: local integral 1} converges \cite[Theorem~2.1]{Liu}.
We note that the well-definedness of \eqref{e: local integral 1} is also shown in Waldspurger~\cite[Section~5.1, Lemme]{Wa2}.
%

Now suppose that $v$ is archimedean.
\begin{Definition}\label{d: archimedean}
For an archimedean place $v$, we define
$\alpha_v\left(\phi_v,\phi^\prime_v\right)$
by  a regularized integral whose regularization is achieved using
the Fourier transform Liu~\cite[3.4]{Liu}.
We refer to \cite[3.4]{Liu} for the details.
\end{Definition}
%
\begin{Remark}\label{r: archimedean}
It is  shown in Liu~\cite[Proposition~3.5]{Liu}
that for \emph{any} place $v$ where $\pi_v$ is square integrable,
the local integral 
\begin{equation}\label{e: local integral 2}
\int_{D_{\lambda,v}}\int_{S_v}
\langle\pi_v\left(s_vt_v\right)\phi_v,\phi_v^\prime\rangle_v\,
\chi_\lambda\left(s_v\right)^{-1}\,ds_v\,dt_v
\end{equation}
does converge absolutely and is equal to $\alpha_v\left(\phi_v,\phi^\prime_v\right)$
defined as above.
We note that later we are only concerned with the case when $\pi_v$ is a discrete series
representation at any archimedean place $v$.
\end{Remark}
%
We recall that the multiplicity one property, i.e.
\begin{equation}
\label{uniqueness}
\dim_{\mathbb C}\operatorname{Hom}_{R_{\lambda,v}}
\left(\pi_v,\chi_{\lambda, v}\right)\le 1,
\end{equation}
holds at \emph{any} place $v$.
As for the proof, we refer to  Gan, Gross and Prasad~\cite[Corollary~15.3]{GGP}
and Jiang, Sun and Zhu~\cite[Theorem~A]{JSZ}
for the non-archimedean case and the archimedean case, respectively.

Moreover when $v$ is non-archimedean, it is shown that 
\begin{multline}\label{e: multiplicity one}
\dim_{\mathbb C}\operatorname{Hom}_{R_{\lambda,v}}
\left(\pi_v,\chi_{\lambda, v}\right)= 1
\\
\Longleftrightarrow
\text{$\alpha_v\left(\phi_v,\phi_v^\prime\right)\ne 0$
for some $\phi_v,\,\phi_v^\prime\in V_v$ which are $K_{G,v}$-finite}
\end{multline}
by 
Waldspurger~\cite[Proposition~5.7]{Wa2}.
It is expected that the equivalence \eqref{e: multiplicity one} holds
also when $v$ is archimedean.
Indeed Beuzzart-Plessis~\cite{BP} proved the corresponding assertion
in the unitary group case
for tempered representations.
We also note that the condition on the right hand side
of \eqref{e: multiplicity one} is equivalent to:
\begin{equation}\label{e: multiplicity one2}
\text{$\alpha_v\left(\phi_v,\phi_v\right)\ne 0$
for some $K_{G,v}$-finite vector
$\phi_v\in V_v$.}
\end{equation}
Indeed $\alpha_v\left(\phi_v,\phi_v^\prime\right)\ne 0$ implies
that the two linear forms $L$ and $L^\prime$ on $V_{\pi_v}$
 defined by
$L\left(\varphi_v\right)=\alpha_v\left(\varphi_v,\phi_v^\prime\right)$
and $L^\prime\left(\varphi_v\right)=
\overline{\alpha_v\left(\phi_v,\varphi_v\right)}$
for $\varphi_v\in V_{\pi_v}$, respectively,
are  non-zero elements of 
$\operatorname{Hom}_{R_{\lambda,v}}
\left(\pi_v,\chi_{\lambda, v}\right)$.
By \eqref{uniqueness}  there exists $c\in \mathbb C^\times$ such that 
$L^\prime=c\cdot L$.
Thus $L^\prime\left(\phi_v\right)=c\cdot L\left(\phi_v\right)=c\cdot
\alpha_v\left(\phi_v,\phi_v^\prime\right)\ne 0$ and hence 
$\alpha_v\left(\phi_v,\phi_v\right)\ne 0$.
%
%
%
%
\subsubsection{Normalization of local integrals}
\label{section: normalization}
We fix maximal compact subgroups $K_G=\prod_v K_{G,v}$ 
of $G\left(\mA\right)$ and $K_\lambda=\prod_v K_{\lambda,v}$
of $D_\lambda\left(\mA\right)$.

A place $v$ is called \emph{good}
(with respect to $\pi$ and a decomposable vector $\phi=\otimes_v\,\phi_v
\in V_\pi=\otimes_v\, V_{\pi_v}$)
if:
%
\begin{subequations}\label{e: good place}
\begin{equation}
\text{$v$ is non-archimedean and is not lying over $2$};
\end{equation}
\begin{equation}
\text{$K_{G,v}$ is a hyperspecial maximal compact subgroup of $G_v$};
\end{equation}
\begin{equation}
\text{$E_v$ is an unramified quadratic extension of $F_v$
or $E_v=F_v\oplus F_v$};
\end{equation}
\begin{equation}
\text{$\pi_v$ is an unramified representation of $G_v$};
\end{equation}
\begin{equation}
\text{$\phi_v$ is a $K_{G,v}$-fixed vector such that $\langle\phi_v,
\phi_v\rangle_v=1$ and $\chi_{\lambda, v}$ is $K_{\lambda, v}$-fixed};
\end{equation}
\begin{equation}
\text{$K_{\lambda, v}\subset K_{G,v}$
and $\mathrm{Vol}\left(K_{G,v},dg_v\right)=\mathrm{Vol}\left(K_{\lambda, v},
dt_v\right)=1$}.
\end{equation}
\end{subequations}
%

Then Liu's theorem~\cite[Theorem~2.2]{Liu} states that when
$v$ is good, one has
\begin{equation}\label{e: liu's theorem}
\alpha_v\left(\phi_v,\phi_v\right)=
\frac{L\left(1/2,\pi_v\right)L\left(1/2,\pi_v\times \chi_{E,v}\right)
\,\prod_{j=1}^n\zeta_{F_v}\left(2j\right)}{
L\left(1,\pi_v, \mathrm{Ad}\right)
L\left(1,\chi_{E,v}\right)}.
\end{equation}
%
\begin{Definition}\label{d: normalized}
We define the normalized  local integral 
$\alpha_v^\natural\left(\phi_v,\phi^\prime_v\right)$
at each place $v$ of $F$
by
\begin{equation}\label{e: normalized}
\alpha_v^\natural\left(\phi_v,\phi^\prime_v\right):=
\frac{L\left(1,\pi_v, \mathrm{Ad}\right)
L\left(1,\chi_{E,v}\right)}{L\left(1/2,\pi_v\right)L\left(1/2,\pi_v\times \chi_{E,v}\right)
\,\prod_{j=1}^n\zeta_{F_v}\left(2j\right)}
\cdot
\alpha_v\left(\phi_v,\phi_v^\prime\right).
\end{equation}

We shall often use the notation
\begin{equation}\label{e: abbreviation}
\alpha_v\left(\phi_v\right):=\alpha_v\left(\phi_v,\phi_v\right)
\quad\text{and}\quad
\alpha_v^\natural\left(\phi_v\right):=\alpha_v^\natural\left(\phi_v,\phi_v\right).
\end{equation}
\end{Definition}
%
\begin{Remark}\label{r: normalized}
Recall that $\zeta_{\mathbb R}\left(s\right)=\pi^{-s\slash 2}\,\Gamma\left(s\slash 2\right)$
and $\zeta_{\mathbb C}\left(s\right)=\left(2\pi\right)^{1-s}\,\Gamma\left(s\right)$.
Here we note that $L\left(s,\pi\right)$ and $L\left(s,\pi\times\chi_{E}\right)$
are defined by the doubling method as in  Lapid and Rallis~\cite{LR}
and  are holomorphic for $\mathrm{Re}\left(s\right)>0$
by Yamana~\cite{Ya}
since $\pi$ is tempered.
It is believed that $L\left(s,\pi,\mathrm{Ad}\right)$
can be analytically continued to the whole $s$-plane,
is holomorphic for $\mathrm{Re}\left(s\right)>0$
and $L\left(1,\pi,\mathrm{Ad}\right)$ is non-zero when
$\pi$ is tempered.
\end{Remark}
%
%
%
%
%
\subsubsection{Refined global Gross-Prasad conjecture on $B_{\lambda, \psi}$}
As in Ichino and Ikeda~\cite{II},
we say that $\pi=\otimes_v\pi_v$ is \emph{almost locally generic}
if the local representation $\pi_v$ is generic at almost all places $v$ of $F$.
Then as explained in \cite[Section~2]{II}, 
such $\pi$ is conjectured to  come from an elliptic Arthur parameter
\[
\Psi\left(\pi\right):\mathcal L_F\to {}^LG:= \hat{G}\rtimes W_F.
\]
Here $\mathcal L_F$ denotes the conjectural Langlands group of $F$
and ${}^LG$ is the Langlands dual group of $G$.
The local representation $\pi_v$ is expected to  be tempered at every $v$
by the generalized Ramanujan conjecture.
Let $\mathcal S\left(\Psi\left(\pi\right)\right)$
 be the centralizer of the image of the Arthur parameter
 $\Psi\left(\pi\right)$ in the complex dual group $\hat{G}$.
For $G\in\mathcal G$, $\mathcal S\left(\Psi\left(\pi\right)\right)$
is a finite elementary $2$-group.
We refer to \cite[2.5]{II} for the details.

The conjecture formulated by Liu~\cite[Conjecture~2.5]{Liu} 
reads as follows, in our case.
%
\begin{conjecture}
Let $\pi=\otimes_v\pi_v$ be an irreducible cuspidal automorphic representation
of $G\left(\mA\right)$ for $G\in\mathcal G$.
Suppose that $\pi$ is 
almost locally generic.
\begin{enumerate}
\item
We have
$\dim_{\mathbb C}\operatorname{Hom}_{R_{\lambda,v}}
\left(\pi_v,\chi_{\lambda, v}\right)= 1$ if and only if
$\alpha_v\left(\phi_v^\prime\right)\ne  0$
for some $K_{G,v}$-finite vector $\phi_v^\prime\in V_{\pi_v}$.
\item
For any non-zero decomposable cusp form
 $\phi=\otimes_v\,\phi_v\in V_\pi$, we have
 \begin{multline}\label{e: Liu's conjecture}
 \frac{\left|B_{\lambda,\psi}\left(\phi\right)\right|^2}{
 \langle\phi,\phi\rangle}
 =\frac{C_\lambda}{\left|\mathcal S\left(\Psi\left(\pi\right)\right)\right|}
 \cdot
 \left(\prod_{j=1}^n\zeta_F\left(2j\right)\right)
 \\
 \times
 \frac{L\left(1/2,\pi\right)L\left(1/2,\pi\times\chi_E\right)
 }{
 L\left(1,\pi,\mathrm{Ad}\right)L\left(1,\chi_E\right)}
\cdot
\prod_v
 \frac{\alpha_v^\natural\left(\phi_v\right)}{
 \langle\phi_v,\phi_v\rangle_v}
 \end{multline}
 where the product is indeed over the finite set of places $v$ of $F$
 which are not good in the sense of \eqref{e: good place}.
 Here   all $L$-functions in \eqref{e: Liu's conjecture}
 denote the completed $L$-functions.
 In particular $\zeta_F\left(s\right)$ denotes the completed Dedekind zeta function of $F$, i.e.
 \begin{equation}\label{e: dedekind}
 \zeta_F\left(s\right)=\prod_{\text{$v:$ place of $F$}}\zeta_{F_v}\left(s\right).
 \end{equation}
  \end{enumerate}
\end{conjecture}
 \begin{Remark}
 When $n=2$ and $G=\mathbb G$, Liu~\cite{Liu},
 inspired by Prasad and Takloo-Bighash~\cite{PT},
  proved
 \eqref{e: Liu's conjecture}
 for endoscopic Yoshida lifts and 
 Corbett~\cite{Co} recently proved it for non-endoscopic Yoshida lifts. 
 We  mention that Qiu~\cite{Qi}  considered a non-tempered case when $n=2$, 
 namely the Saito-Kurokawa lifting case.
 We also mention that 
  Murase and Narita~\cite{MN} proved an explicit formula
 for the central $L$-values in terms of the Bessel periods
 for Arakawa lifts
 when $n=2$ and $G$ is not split.
 \end{Remark}
%
%
%
\subsection{Main Theorem}\label{ss: main theorem}
We say that an irreducible cuspidal tempered automorphic representation $\pi=\otimes_v
\,\pi_v$
 of $G\left(\mA\right)$ for $G\in\mathcal G$
 has a \emph{weak lift}  to $\gl_{2n}\left(\mA\right)$
 if there exists an irreducible automorphic representation
 $\Pi=\otimes_v\,\Pi_v$ of $\gl_{2n}\left(\mA\right)$
 such that $\Pi_v$ is a local Langlands lift of $\pi_v$
at almost all non-archimedean places and
all archimedean places.
 If such $\Pi$ exists,  it is unique by the classification theorem of Jacquet and Shalika~\cite[(4.4)]{JS}
 and  is written as
an isobaric sum
\begin{equation}\label{e: isobaric}
\Pi=\boxplus_{i=1}^l\,\pi_i
\end{equation}
where $\pi_i$ is an irreducible cuspidal automorphic representation of 
$\gl_{2n_i}\left(\mA\right)$ such that:
\begin{equation*}\label{e: weak lift-a}
\text{$L\left(s,\pi_i, \wedge^2\right)$
has a pole at $s=1$},\quad\sum_{i=1}^l n_i=n,
\quad\text{$\pi_i\not\simeq\pi_j$ for $i\ne j$}.
\end{equation*}
%
When $G=\mathbb G$, the existence of 
a weak lift  is guaranteed by 
Arthur~\cite[Theorem~1.5.2]{Ar}. 
 
Our aim in this paper is to prove the following theorem.
%
%
%
%
\begin{theorem}\label{t: main theorem}
Let $F$ be a totally real number field
and $\pi=\otimes_v\,\pi_v$  an irreducible cuspidal tempered automorphic representation
 of $G\left(\mA\right)$ for $G\in\mathcal G$
 such that $\pi_v$ is a discrete series representation
 at any archimedean place $v$ of $F$.
 
 Suppose that the special Bessel period $B_{\lambda, \psi}$
 of type $E$ does not vanish identically on the space of cusp forms
 $V_\pi$ for $\pi$.
 Let $\Pi$ be a weak lift of $\pi$ to $\gl_{2n}\left(\mA\right)$,
which is written 
of the form ~\eqref{e: isobaric}.

Then the following assertions hold.
 %
 \begin{enumerate}
 \item
 At each place $v$, there exists a $K_{G,v}$-finite vector $\phi_v^\prime
 \in V_{\pi_v}$ such that $\alpha_v\left(\phi_v^\prime\right)\ne 0$.
 \item
 For any non-zero decomposable cusp form 
 $\phi=\otimes_v\,\phi_v\in V_\pi$, we have
 %
 \begin{multline}\label{e: main identity}
 \frac{\left|B_{\lambda,\psi}\left(\phi\right)\right|^2}{
 \langle\phi,\phi\rangle}
 =2^{-l}\,C_\lambda
 \cdot
 \left(\prod_{j=1}^n\zeta_F\left(2j\right)\right)
 \\
 \times
 \frac{L\left(1/2,\pi\right)L\left(1/2,\pi\times\chi_E\right)
 }{
 L\left(1,\pi,\mathrm{Ad}\right)L\left(1,\chi_E\right)}
 \cdot\prod_v
 \frac{\alpha_v^\natural\left(\phi_v\right)}{
 \langle\phi_v,\phi_v\rangle_v}.
 \end{multline}
 \end{enumerate}
 %
 %
 Here $L\left(s,\pi,\mathrm{Ad}\right)$ is defined as
 $L\left(s,\pi, \mathrm{Ad}\right)=\prod_v
 L\left(s,\pi_v,\mathrm{Ad}\right)$ where
 \begin{equation}\label{sym=ad}
 L\left(s,\pi_v,\mathrm{Ad}\right):=L\left(s,\Pi_v,\mathrm{Sym}^2\right)
 \end{equation}
 for each place $v$ and $\mathrm{Sym}^2$ denotes
 the symmetric square representation of 
 $\gl_{2n}\left(\mathbb C\right)$.
 \end{theorem}
 %
 %
 %
 \begin{Remark}
 The existence of a weak lift $\Pi$ readily follows from our previous paper \cite[Theorem~1]{FM0}, as explained in
 the beginning of \ref{ss: set up}.
 \end{Remark}
 \begin{Remark}
 When $\pi$ has a weak lift $\Pi$ to $\gl_{2n}\left(\mA\right)$
 of the form \eqref{e: isobaric}, it is clear from the definition
 of the Arthur parameter that
 $2^l=\left|\mathcal S\left(\Psi\left(\pi\right)\right)\right|$.
 \end{Remark}
 %
\begin{Remark}
Suppose that $\pi_v$ is unramified.
Then we may define $L\left(s,\pi_v,\mathrm{Ad}\right)$ 
in terms of the Satake parameter of $\pi_v$.
This coincides with the one defined by \eqref{sym=ad}.
 \end{Remark}
 %
 %
 %
 %
%
%
%
%
The following corollary is proved in Section~\ref{s: maincor}.
\begin{corollary}\label{maincor}
Keep the assumption in Theorem~\ref{t: main theorem}
except for $B_{\lambda,\psi}\not\equiv 0$ on $V_\pi$.
If we assume that 
Arthur's conjectures \cite[Conjecture~9.4.2, Conjecture~9.5.4]{Ar} hold
for any $G^\prime\in\mathcal G$,
the equality \eqref{e: main identity} holds
for any non-zero decomposable cusp form $\phi=\otimes_v\,\phi_v
\in V_\pi$.
\end{corollary}
%
We refer to Ichino and Ikeda~\cite[2.5]{II}
for the relevance of Arthur's conjectures to the Gross-Prasad conjecture.
%
%
%

For the sake of the reader,
here we explain  the skeleton of our proof of 
\eqref{e: main identity}.
As in our previous paper \cite{FM0}, the theta correspondence
between $G\in\mathcal G_n$ and $\widetilde{\operatorname{Sp}}_n$,
i.e. rank $n$ metaplectic group, 
plays a pivotal role.

Suppose that $B_{\lambda,\psi}\left(\phi\right)\ne  0$ for $\phi\in V_\pi$.
The computation in \cite{Fu} of the pull-back of the $\psi_\lambda$-Whittaker period
$W\left(\theta_\psi^\varphi\left(\phi\right);\psi_\lambda\right)$,
which is defined by \eqref{e: metaplectic Whittaker},
of $\theta_\psi^\varphi\left(\phi\right)$, the theta
lift 
of $\phi$ to $\widetilde{\operatorname{Sp}}_n
\left(\mA\right)$ with respect to the additive character $\psi$
and the test function $\varphi$, yields
\begin{equation}\label{e: B=W crude}
B_{\lambda,\psi}\left(\phi\right)
\aaeq 
C_G^{-1}C_\lambda\cdot
W\left(\theta_\psi^\varphi\left(\phi\right);\psi_\lambda\right).
\end{equation}
Here we use the symbol ``$\aaeq$'' to imply that the two sides are
equal up to multiplication by a product of finitely many local factors.
Then  the remarkable formula obtained by Lapid and Mao~\cite{LM}
implies that
we have
\begin{equation}\label{e: LM crude}
\frac{\,\left| W\left(\theta_\psi^\varphi\left(\phi\right);\psi_\lambda\right)\right|^2\,}{
\left(\theta_\psi^\varphi\left(\phi\right),\theta_\psi^\varphi\left(\phi\right)\right)}
\aaeq
2^{-l}\cdot
\frac{L\left(1/2,\pi\times\chi_E\right)\prod_{j=1}^n \zeta_F\left(2j\right)}{
L\left(1,\pi,\operatorname{Ad}\right)}
\end{equation}
where $\left(\theta_\psi^\varphi\left(\phi\right),
\theta_\psi^\varphi\left(\phi\right)\right)$
is the square of the Petersson norm of $\theta_\psi^\varphi\left(\phi\right)$.
On the other hand, 
some proper adjustments to the proof of the \emph{precise Rallis inner product formula}
in Gan and Takeda~\cite{GT1}
yield one in our case, namely
\begin{equation}\label{e: GT crude}
\frac{\left(\theta_\psi^\varphi\left(\phi\right),\theta_\psi^\varphi\left(\phi\right)\right)}{
\langle\phi,\phi\rangle}
\aaeq
C_G\cdot
\frac{L\left(1/2,\pi\right)}{\prod_{j=1}^n\zeta_F\left(2j\right)}.
\end{equation}
Since we have
$C_G\cdot\prod_{j=1}^n\zeta_F\left(2j\right)
= C_\lambda\cdot L\left(1,\chi_E\right)$, 
the combination of 
\eqref{e: B=W crude}, \eqref{e: LM crude}
and \eqref{e: GT crude} yields
%
\begin{equation}\label{e: GP crude}
\frac{\left|B_{\lambda,\psi}\left(\phi\right)\right|^2}{
\langle\phi,\phi\rangle}\aaeq
2^{-l}\cdot C_\lambda\cdot
\frac{L\left(1/2,\pi\right)L\left(1/2,\pi\times\chi_E\right)
\prod_{j=1}^n\zeta_F\left(2j\right)
}{
L\left(1,\pi,\operatorname{Ad}\right)L\left(1,\chi_E\right)}.
\end{equation}
%
Thus our task   is to  elaborate \eqref{e: GP crude} to
the precise equality  \eqref{e: main identity} by executing  the above
idea
rigorously and proving a certain local equality
by some intricate arguments.
It is done in Section~\ref{s: local identity},
which is the heart of the matter of this paper.
%
%
%
%
%
%
%
%
%
%
%
%
%
%
\subsection{B\"ocherer's conjecture}
By considering the case when $n=2$, $F=\mathbb Q$ and $G=\mathbb G_2\simeq 
\operatorname{PGSp}_2$,
Theorem~\ref{t: main theorem} yields a proof of the long-standing conjecture
of B\"ocherer~\cite{Bo} concerning 
central critical values of imaginary quadratic twists of spinor $L$-functions
for holomorphic Siegel cusp forms of degree two which are Hecke eigenforms,
thanks to the recent work of Dickson, Pitale, Saha and Schmidt~\cite{DPSS}.
Namely  B\"ocherer's conjecture holds in the following refined form.
%
%
%
%
\begin{theorem}\label{t: Boecherer}
Let $\varPhi$ be a holomorphic Siegel cusp form of degree two and weight $k$
with respect to $\operatorname{Sp}_2\left(\mathbb Z\right)$
which is a Hecke eigenform and $\pi\left(\varPhi\right)$
the associated automorphic representation 
of $\operatorname{PGSp}_2\left(\mA\right)\simeq
\mathbb G_2\left(\mA\right)$.
Suppose that $\varPhi$ is not a Saito-Kurokawa lift.
Let 
\[
\varPhi\left(Z\right)=\sum_{T>0}a\left(T,\varPhi\right)
\exp\left[2\pi\sqrt{-1}\operatorname{tr}\left(TZ\right)\right],
\,\, Z\in\mathfrak H_2,
\]
be the Fourier expansion where $T$ runs over semi-integral positive definite 
two by two symmetric matrices 
and $\mathfrak H_2$ denotes the Siegel upper half space of degree two.

For an imaginary quadratic field $E$ with discriminant 
$-D_E$, let us define $B\left(\varPhi;E\right)$ by
\[
B\left(\varPhi;E\right):=w\left(E\right)^{-1}
\sum_{\left\{T: \det T=D_E\slash 4\right\}\slash \sim}
a\left(T,\varPhi\right)
\]
where $\sim$ denotes the equivalence relation defined by $T_1\sim T_2$
if there exists an element 
$\gamma$ of $\operatorname{SL}_2\left(\mathbb Z\right)$
such that ${}^t\gamma T_1\gamma = T_2$
and 
$w\left(E\right)$ is the number of roots of unity in $E$.
We recall that when $\det T=D_E\slash 4$, the number of elements
in $\left\{\gamma\in\operatorname{SL}_2\left(\mathbb Z
\right): {}^t\gamma T\gamma =T\right\}$
is equal to $w\left(E\right)$.

Then we have
\begin{equation}\label{e: Boecherer}
\frac{\left|B\left(\varPhi; E\right)\right|^2}{
\langle\varPhi,\varPhi\rangle}
=2^{2k-4}\cdot
 D_E^{k-1}\cdot
\,
\frac{L\left(1/2,\pi\left(\varPhi\right)\right)
L\left(1/2,\pi\left(\varPhi\right)\times\chi_E\right)}{
L\left(1,\pi\left(\varPhi\right),\operatorname{Ad}\right)}.
\end{equation}
Here 
\[
\langle\varPhi,\varPhi\rangle=\int_{\mathrm{Sp}_2\left(\mathbb Z\right)
\backslash \mathfrak H_2}
\left|\varPhi\left(Z\right)\right|^2\det\left(Y\right)^{k-3}\,dX\, dY
\quad\text{where $Z=X+\sqrt{-1}\,Y$.}
\]
\end{theorem}
%
\begin{proof}
First we note that $\pi\left(\Phi\right)$ is tempered
by Weissauer~\cite{We2}
since $\Phi$ is not a Saito-Kurokawa lift,
as explained 
in the proof of \cite[Theorem~4]{FM0}.

Recall that
it is shown in our previous paper~\cite[Theorem~5]{FM0}
that:
\begin{equation}\label{e: weak form}
B\left(\varPhi;E\right)\ne 0 \Longleftrightarrow
L\left(1/2,\pi\left(\varPhi\right)\right)L\left(1/2,\pi\left(\varPhi\right)\times
\chi_E\right)\ne 0.
\end{equation}

When $k$ is odd, we have $L\left(1/2,\pi\left(\varPhi\right)\right)=0$
by Andrianov~\cite[Theorem~3.1.1. (II)]{An}.
Hence \eqref{e: Boecherer} holds by \eqref{e: weak form}.
We mention that $B\left(\varPhi;E\right)=0$ 
also follows in a more elementary way from
${}^t\gamma \left\{T: \det T=D_E\slash 4\right\}\gamma
=\left\{T: \det T=D_E\slash 4\right\}$ and
$a\left({}^t\gamma T\gamma,\varPhi\right)=
\left(\det\gamma\right)^k\cdot a\left( T,\varPhi\right)$
for $\gamma\in\mathrm{GL}_2\left(\mathbb Z\right)$,
as remarked in B\"ocherer~\cite[p.31]{Bo}.

Suppose that $k$ is even.
If $B\left(\varPhi;E\right)=0$,
\eqref{e: Boecherer} holds by \eqref{e: weak form}.
If $B\left(\varPhi;E\right)\ne 0$,
\eqref{e: Boecherer} follows from \eqref{e: main identity}
by Dickson et al. \cite[1.12~Theorem]{DPSS}.
\end{proof}
%
\begin{Remark}
In \cite{Bo},
B\"ocherer conjectured that there exists a constant
$c_\varPhi$ which depends only on $\varPhi$ such that we have
\[
\left|B\left(\varPhi;E\right)\right|^2=
c_\varPhi\cdot  D_E^{k-1}\cdot 
L\left(1/2,\pi\left(\varPhi\right)\times\chi_E\right)
\]
for any imaginary quadratic  field $E$.
As far as we know, B\"ocherer did not speculate on the constant
$c_\varPhi$ except when $\varPhi$ is a Saito-Kurokawa lift.
For the exact formula for the left hand side of
\eqref{e: Boecherer}
when $\varPhi$ is a Saito-Kurokawa lift, we refer to Dickson et al.~\cite[3.12~Theorem]{DPSS}.
\end{Remark}
%
\begin{Remark}
For brevitiy, only the full modular case is stated in 
Theorem~\ref{t: Boecherer}.
In fact, Theorem~\ref{t: main theorem} yields 
\cite[1.12~Theorem]{DPSS} unconditionally when $\Lambda=1$, i.e.
the refined form of B\"ocherer's conjecture has an extension
to   the odd
square free  level case.

We expect Theorem~\ref{t: Boecherer} and its extension
to have a broad spectrum of interesting consequences, e.g.
\cite[Section~3.6]{DPSS}. We also mention 
Blomer~\cite{Blomer} and Kowalski, Saha and Tsimerman~\cite{KST}.
It is expected that the extension of Theorem~\ref{t: Boecherer} holds
also in the case when $k=2$, which we wish to consider in the near future.
\end{Remark}
%
%
%
\begin{Remark}
In the sequel, we shall pursue the generalization of 
Theorem~\ref{t: Boecherer} and its extension
to the case when the character of the ideal class group  is  not necessarily trivial.
We shall also pursue the case when the Siegel cusp form in question
is vector valued.
\end{Remark}
%
%
%
%
%
%
%
\subsection{Organization of the paper}
This paper is organized as follows.
In Section~\ref{s: reduction},  first we  review
the precise Rallis inner product formula by
Gan and Takeda~\cite{GT1} and the
explicit formula for the Whittaker periods on
the metaplectic group by Lapid and Mao~\cite{LM}.
Both formulas play decisive roles as explained in \ref{ss: main theorem}.
After the review, we turn to the proof of Theorem~\ref{t: main theorem}.
By combining these two formulas with the pull-back formula for the
Whittaker periods on the metaplectic group in \cite{Fu},
the proof of Theorem~\ref{t: main theorem}  is reduced to verifying
a certain local equality, which we prove in
Section~\ref{s: local identity}.
Then in Section~\ref{s: maincor}
we deduce Corollary~\ref{maincor}
from Theorem~\ref{t: main theorem}, 
assuming Arthur's conjectures.
%
%
%
\subsection*{Acknowledgements}
It is evident that this paper owes its existence to the Ichino-Ikeda type formula
for Whittaker periods on metaplectic groups
proved by Lapid and Mao, as the culmination of their
profound  papers~\cite{LM-1,LM0,LM1,LM}.
The authors  would like to express their  admiration
for Lapid and Mao.
Indeed in \cite{LM0} they  formulate the Ichino-Ikeda type conjectural
formula for Whittaker periods on general quasi-split reductive groups,
besides metaplectic groups.
We expect their conjectural formula to have broad and deep 
influence on future research in the field of automorphic forms.
It is also evident that, as in our previous paper~\cite{FM0},
 we take full advantage of the recent
significant
contributions to theta correspondence.
 The most notable ones in this paper are \cite{GT,GT1}
by Gan and Takeda.
The former~\cite{GT}  allows us to use freely the global-local machinery 
and the latter~\cite{GT1}  is one of the main ingredients for the proof of 
\eqref{e: main identity}, as explained above.
We also mention that
not only the formulation of the conjecture by Liu~\cite{Liu} 
gave us the starting point of this work
but also our arguments in Section~\ref{s: local identity}
are very much inspired by his arguments in \cite[3.5]{Liu}.
%

The first author has been pursuing the relative trace formula approach towards
B\"ocherer's conjecture with Martin and Shalika \cite{FS,FM-1,FMS,FM-2}.
Although some substantial progress has been made, 
one must admit that there still lies ahead
 a long way to go by this approach.
The remarkable result by Lapid and Mao enabled us to present a proof 
of B\"ocherer's conjecture here
in the special Bessel period case.
The first author would like to thank Kimball Martin for
prompting him to reexamine the theta correspondence approach.

The authors are very grateful to  Atsushi Ichino for  important suggestions
and encouragements.
They would like to heartily thank Martin Dickson, Ameya Pitale, Abhishek
Saha and Ralf Schmidt for helpful  and illuminating correspondences
concerning their paper~\cite{DPSS}.
Thanks are also due to Shunsuke Yamana for kindly pointing out an  oversight
concerning  the Rallis inner product formula
in an earlier version of this paper.

Last but not least, the authors would like to express their deep gratitudes
to the anonymous referee for carefully reading the manuscript and offering many invaluable 
suggestions.
%
%
%
%
%
%
%
%
%
%
%
%
%
%
%
%
\section{Reduction to a local equality}\label{s: reduction}
%
%
%
%
\subsection{Set up}\label{ss: set up}
Throughout this section, $\pi=\otimes_v\pi_v$ is an irreducible cuspidal tempered automorphic representation of $G\left(\mathbb A\right)$
for $G=\so\left(V\right)\in\mathcal G_n$
over a totally real number field $F$ such that:
\begin{subequations}\label{e: conditions on pi}
\begin{equation}\label{e: discreteness}
\text{$\pi_v$ is a discrete series representation at any real place $v$};
\end{equation}
%
\begin{equation}\label{e: non-zero Bessel}
\text{$B_{\lambda, \psi}$, the Bessel model 
of type $E$, does not vanish identically on $V_\pi$}.
\end{equation}

Since $\pi$ is tempered, 
by Remark~2 in \cite{FM0},
Theorem~1 in \cite{FM0} and the arguments in the course
of its proof  are all applicable to $\pi$.
Hence we have
\begin{equation}\label{e: non-vanishing}
L\left(1/2,\pi\right)\cdot L\left(1/2,\pi\times\chi_E\right)\ne 0
\end{equation}
and there exists a globally generic irreducible cuspidal automorphic
representation $\pi^\circ$ of $\mathbb G\left(\mA\right)$
which is nearly equivalent to $\pi$.
Thus
\begin{equation}\label{e: locally almost generic}
\text{$\pi$ is almost locally generic.}
\end{equation}
We note that when 
$n=2$,
$F=\mathbb Q$ and $G=\mathbb G$,
the existence of such $\pi^\circ$ also follows
from Weissauer~\cite{We1}.

Since $\pi^\circ$ has a weak lift to $\mathrm{GL}_{2n}\left(\mA\right)$
by Arthur~\cite[Theorem~1.5.2]{Ar},
we may say:
\begin{equation}\label{e: weak lift}
\text{$\pi$ has a weak lift $\Pi$ to $\gl_{2n}\left(\mA\right)$
of the form \eqref{e: isobaric}}.
\end{equation}
We note that the existence of a weak lift of $\pi^\circ$
to $\mathrm{GL}_{2n}\left(\mA\right)$
also follows from Cogdell, Kim, Piatetski-Shapiro and Shahidi~\cite{CKPSS}
since $\pi^\circ$ is generic.
%

For a positive integer $n$, let $Y_n$ be the space of $2n$-dimensional
row vectors equipped with the alternating form
\[
\langle w_1,w_2\rangle=w_1 \begin{pmatrix}0&1_n\\-1_n&0\end{pmatrix}
{}^tw_2\quad\text{for $w_1,w_2\in Y_n$}.
\]
Let $Y_n=Y_n^+\oplus Y^-_n$ be the polarization
where
\[
Y_n^+:=\left\{\left(y_1,\cdots , y_{2n}\right):
y_i=0\,\left(1\le i\le n\right)\right\}
\]
and
\[
Y_n^-:=\left\{\left(y_1,\cdots , y_{2n}\right):
y_i=0\,\left(n+1\le i\le 2n\right)\right\}.
\]
Let $\symp_n$ denote the rank $n$ symplectic group defined by
\[
\symp_n:=\left\{g\in\gl_{2n}:
g\begin{pmatrix}0&1_n\\-1_n&0\end{pmatrix}{}^tg=
\begin{pmatrix}0&1_n\\-1_n&0\end{pmatrix}\right\}
\]
which acts on $Y_n$ from the right.
We recall that $\widetilde{\symp}_n\left(\mA\right)$,
the rank $n$ metaplectic group over $\mA$,  is a certain twofold central extension
of $\symp_n\left(\mA\right)$.
The theta correspondence of automorphic forms
between $\widetilde{\symp}_n\left(\mA\right)$
and $G\left(\mA\right)$
for $G\in\mathcal G_n$ plays the essential role  as in
our previous paper~\cite{FM0}.
%
%
%

Let us realize $\omega_\psi=\omega_{\psi, V, Y_n}$,
the Weil representation of 
$G\left(\mA\right)\times \widetilde{\symp}_n\left(\mA\right)$
with respect  to $\psi$,
 on $\mathcal S\left(\left(V\otimes Y_n^+\right)
\left(\mA\right)\right)$,
the Schwartz-Bruhat space on $\left(V\otimes Y_n^+\right)
\left(\mA\right)$,
by taking $V\otimes Y_n=\left(V\otimes Y_n^+\right)\oplus
\left(V\otimes Y_n^-\right)$ 
as a polarization of the symplectic space $V\otimes Y_n$.
For $\phi\in V_\pi$ and 
$\varphi\in\mathcal S\left(\left(V\otimes Y_n^+\right)
\left(\mA\right)\right)$, 
the theta lift $\theta_\psi^\varphi\left(\phi\right)$
of $\phi$ to $\widetilde{\symp}_n\left(\mA\right)$
with respect to the additive character $\psi$
and the test function $\varphi$ is defined by
\[
\theta_\psi^\varphi\left(\phi\right)\left(h\right):=
\int_{G\left(F\right)\backslash G\left(\mA\right)}
\left(\sum_{z\in \left(V\otimes Y_n^+\right)
\left(F\right)}
\left(\omega_\psi\left(g,h\right)\varphi\right)\left(z\right)\right)
\phi\left(g\right)dg
\quad\text{for $h\in\widetilde{\symp}_n\left(\mA\right)$.}
\]
Let $\Theta_n\left(\pi,\psi\right)$ denote
the automorphic representation of $\widetilde{\symp}_n\left(\mA\right)$
generated by $\theta_\psi^\varphi\left(\phi\right)$
where $\phi$ and $\varphi$ vary in $V_\pi$ and 
$\mathcal S\left(\left(V\otimes Y_n^+\right)
\left(\mA\right)\right)$, respectively.
For the sake of simplicity, we shall write $\sigma$ for $\Theta_n\left(\pi,\psi\right)$
and $V_\sigma$ for its space of automorphic forms.

Then by the  proof of Theorem~1 in \cite{FM0}, we have:
%
\begin{equation}\label{e: theta lift}
\text{$\sigma=\Theta_n\left(\pi,\psi\right)$ is $\psi_\lambda$-generic,
irreducible and cuspidal.}
\end{equation}
%
\end{subequations}
We refer to \ref{ss: lapid mao}
for the definition of $\psi_\lambda$-genericity.
%
%
%
%
\subsection{Rallis inner product formula}
Gan and Takeda~\cite{GT1} proved the precise
Rallis inner product formula in the 
$(\mathrm{O}_{2n+1},\widetilde{\mathrm{Sp}}_n)$ setting.
On the other hand, the one we  need for our purpose is in the 
$(\mathrm{SO}_{2n+1},\widetilde{\mathrm{Sp}}_n)$ setting.
Here we recall the Rallis inner product formula, 
with some explanations of  adjustments to the proof
in \cite{GT1} necessary to deduce the one in our setting.

Let $\mathbb W$ be the quadratic space $V\oplus \left(-V\right)$,
i.e. as a vector space $\mathbb W$ is a direct sum $V\oplus V$
and its symmetric bilinear form $\left(\, \,, \,\,\right)_{\mathbb W}$
on $\mathbb W$
is defined by
\[
\left(v_1\oplus v_2,v_1^\prime\oplus v_2^\prime\right)_{\mathbb W}
:=\left(v_1,v_1^\prime\right)_{V}-\left(v_2,v_2^\prime\right)_V.
\]
Let $\mathbb W^+$ be a maximal isotropic subspace of $\mathbb W$
defined by
\[
\mathbb W^+:=\left\{v\oplus v\in\mathbb W: v\in V\right\}.
\]
We note that there is a natural embedding
\[
\iota:\mathrm{SO}\left(V\right)\times\mathrm{SO}\left(-V\right)
\hookrightarrow
\mathrm{SO}\left(\mathbb W\right)
\quad
\text{such that $\iota\left(g_1,g_2\right)\left(v_1\oplus v_2\right)=
g_1v_1\oplus g_2 v_2$}.
\]
Also there exists an $\mathrm{SO}(V, \mA) \times \mathrm{SO}(-V, \mA)$-intertwining map
\[
\tau: \mathcal{S}\left(\left(V \otimes Y^+_n\right)\left(\mA\right)\right) \hat{\otimes}\,
 \mathcal{S}\left(\left(\left(-V\right) \otimes Y^+_n\right)\left(\mA\right)\right)
 \rightarrow 
 \mathcal{S}\left(\left(\mathbb{W}^+ \otimes Y_n\right)\left(\mA\right)\right)
\]
with respect to the Weil representations,
obtained by composing the natural map
\[
\mathcal{S}\left(\left(V \otimes Y^+_n\right)\left(\mA\right)\right) \hat{\otimes}\,
 \mathcal{S}\left(\left(\left(-V\right) \otimes Y^+_n\right)\left(\mA\right)\right)
 \rightarrow 
 \mathcal{S}\left(\left(\mathbb{W} \otimes Y_n^+\right)\left(\mA\right)\right)
 \]
 with the partial Fourier transform
\[
\mathcal{S}\left(\left(\mathbb{W} \otimes Y_n^+\right)\left(\mA\right)\right)
\xrightarrow{\sim}
\mathcal{S}\left(\left(\mathbb{W}^+ \otimes Y_n\right)\left(\mA\right)\right).
\]
 Namely we have 
 \[
 \tau\left(\omega_{\psi, V, Y_n}\left(g_1\right)\varphi_{+}
 \otimes\omega_{\psi, -V, Y_n}\left(g_2\right)\varphi_{-}\right)
 =\omega_{\psi, \mathbb W, Y_n}\left(\iota\left(g_1,g_2\right)\right)
 \tau\left(\varphi_{+}\otimes\varphi_{-}\right)
 \]
 for $\left(g_1,g_2\right)\in\mathrm{SO}(V, \mA) \times \mathrm{SO}(-V, \mA)$
 and
 $\varphi_{\pm}\in\mathcal{S}\left(\left(\left(\pm
 V\right) \otimes Y^+_n\right)\left(\mA\right)\right)$
 by the double sign.
 
 We also consider local counterparts of $\tau$.
 %
 %
 %
 \subsubsection{Local doubling zeta integrals}
 Let $P$ be the maximal parabolic subgroup of
 $\so\left(\mathbb W\right)$
 defined as the stabilizer of the isotropic subspace
 $V\oplus\left\{0\right\}$.
 Then the Levi subgroup of $P$
 is isomorphic to $\gl\left(V\right)$.

 At each place $v$ of $F$, we consider the degenerate principal 
 representation 
 \[
 I_v\left(s\right):=\mathrm{Ind}_{P\left(F_v\right)}^{
 \so\left(\mathbb W, F_v\right)}\,|\,\,\,\,|_v^s
 \quad\text{for $s\in\mathbb C$.}
 \]
 Here the induction is normalized and
  $|\,\,\,\,|_v^s$ denotes the character of $P\left(F_v\right)$
  which is given by $|\det |_v^s$ on its Levi subgroup $\gl\left(V, F_v\right)$
  and is trivial on its unipotent radical.
  
  For $\phi_v,\phi_v^\prime\in V_{\pi_v}$
  and $\Phi_v\in I_v\left(s\right)$, the local doubling zeta integral
    is defined by
  \begin{equation}\label{e: local doubling}
  Z_v\left(s,\phi_v,\phi_v^\prime,\Phi_v,\pi_v\right):=
  \int_{G_v}\langle\pi_v\left(g_v\right)\phi_v,\phi_v^\prime\rangle_v
  \,\Phi_v\left(\iota\left(g_v,e_v\right)\right)\,dg_v
 \end{equation}
 where $e_v$ denotes the unit element of $G_v$.
 We recall that the integral~\eqref{e: local doubling} converges absolutely
 when $\mathrm{Re}\left(s\right)> -\frac{1}{2}$ by
 Yamana~\cite[Lemma~7.2]{Ya}
 since $\pi_v$ is tempered.
 
 For $\varphi_v\in\mathcal{S}\left(\left(\mathbb{W}^+ \otimes Y_n\right)
 \left(F_v\right)\right)$, we define $\Phi_{\varphi_v}
 \in I_v\left(0\right)$
 by
 \begin{equation}
 \label{local section}
 \Phi_{\varphi_v}\left(g_v\right)=
 \left(\omega_{\psi_v}\left(g_v\right)\varphi_v\right)
 \left(0\right)
 \quad\text{for $g_v\in\so\left(\mathbb W, F_v\right)$}.
 \end{equation}
 %
 \begin{Definition}\label{d: inner product formula ingredient}
 We define $Z_v^\circ\left(\phi_v, \varphi_v,\pi_v\right)$
 for $\phi_v\in V_{\pi_v}$ and $\varphi_v \in
 \mathcal{S}\left(\left(V \otimes Y_n^+\right)
 \left(F_v\right)\right)$ by
 \begin{equation}\label{e: normalized doubling}
Z_v^\circ\left(\phi_v, \varphi_v,\pi_v\right):
=
\frac{\prod_{j=1}^n\zeta_{F_v}\left(2j\right)}{L\left(1\slash 2, \pi_v\right)}
\cdot \frac{1}{\langle\phi_v,\phi_v\rangle_v}\cdot
Z_v\left(0, \phi_v,\phi_v,\Phi_{\tau_v\left(\varphi_v\otimes\varphi_v\right)},
\pi_v\right).
\end{equation}
\end{Definition}
%
%
%
\subsubsection{Rallis inner product formula}   
%
For automorphic forms $\tilde{\varphi}_1$ and $\tilde{\varphi}_2$
on $\widetilde{\mathrm{Sp}}_n\left(\mA\right)$, 
we define the Petersson inner product 
$\left(\tilde{\varphi}_1,\tilde{\varphi}_2\right)$ by
\[
\left(\tilde{\varphi}_1,\tilde{\varphi}_2\right):=
\int_{\symp_n\left(F\right)\backslash \symp_n\left(\mA\right)}
\tilde{\varphi}_1\left(g\right)
\overline{\tilde{\varphi}_2\left(g\right)}\, dg
\]
when the integral converges absolutely.
We recall that $dg$ is the Tamagawa measure.

Let us first recall the Siegel-Weil formula.
Let $\Phi$ be a standard section of 
$\mathrm{Ind}_{P(\mA)}^{\mathrm{SO}(\mathbb W)(\mA)} |\,\,\,|^s$.
Then we form the Siegel Eisenstein series by
\[
E(g, s; \Phi) = \sum_{\gamma \in P(F) \backslash \mathrm{SO}(\mathbb{W})(F)} \Phi(\gamma g, s).
\]
This sum converges absolutely when $\mathrm{Re}(s) > n$ and it has a meromorphic continuation to $\mC$.
We note that 
our Eisenstein series slightly differs from  the Eisenstein series
$E^{\left(m,m\right)}$ in \cite[p.183]{GT1} for $m=2n+1$
since $P$ is also the Siegel parabolic subgroup of $\mathrm{O}(\mathbb W)$
and $P \backslash \mathrm{O}(\mathbb W) \ne P \backslash \mathrm{SO}(\mathbb W)$.

Let 
\[
E(g, s; \Phi)  = \sum_{d \geq -1} A_d(\Phi)(g)s^d
\]
be the Laurent expansion of $E(g, s; \Phi)$ at $s=0$
where $A_d(\Phi)$ is an automorphic form on $\mathrm{SO}(\mathbb W)(\mA)$.
For $\varphi \in\mathcal{S}\left(\left(\mathbb{W}^+ \otimes Y_n\right)
 \left(\mA \right)\right)$, we define the section $\Phi_{\varphi} \in \mathrm{Ind}_{P(\mA)}^{\mathrm{SO}(\mathbb W)(\mA)} |\,\,|^s$ as in the local situation \eqref{local section}
 and 
we simply write $A_d\left(\varphi\right)$ for $A_d\left(\Phi_\varphi\right)$.
 
On the other hand, let $I^{(2n+1, n)}(g, s; \varphi)$ be the regularized theta integral defined in \cite[p.185]{GT1}.
Then we have an equality
\[
I^{(2n+1, n)}(g, s; \omega(z)\varphi) = P_z(s) \mathcal{E}^{(2n+1, n)}(g, s; \varphi)
\]
as in \cite[p.186]{GT1}.
Here $z$ is a regularizing element given in \cite[p.185]{GT1},
$P_z(s)$ is a certain holomorphic function in $s$ depending on $z$
and $\mathcal{E}^{(2n+1, n)}(g, s; \varphi)$ is a certain Eisenstein series 
defined in \cite[p.186]{GT1}.
Let us write the Laurant expansion of $\mathcal{E}^{(2n+1, n)}(g, s; \varphi)$ at 
$s=\frac{n+1}{2}$
as 
\[
\mathcal{E}^{(2n+1, n)}(g, s; \varphi) = \sum_{d \geq -2} B_d^{(2n+1, n)}(\varphi) \left(s-\frac{n+1}{2} \right)^d.
\]

Let $\varphi^0$ denote the spherical Schwartz function defined in \cite[p.181]{GT1}.
We denote by 
$\mathcal{S}^\circ\left(\left(\mathbb{W}^+ \otimes Y_n\right) \left(\mA\right)\right)$ 
the $\mathrm{SO}(\mathbb{W})(\mA)$-span of $\varphi^0$.
Then we have the following Siegel-Weil formula.
%
\begin{proposition}
\label{SW}
With the above notation, we have 
\[
A_0(\varphi) = B_{-1}^{(2n+1, n)}(\varphi)
\]
for any $\varphi \in \mathcal{S}^\circ\left(\left(\mathbb{W}^+ \otimes Y_n\right) \left(\mA\right)\right).$
\end{proposition}
%
\begin{proof}
The argument for the proof of  \cite[Proposition~5.3]{GT1} for
$\left(\mathrm{O}_{2n+1}, \widetilde{\mathrm{Sp}}_n\right)$
works, mutatis mutandis, for 
$\left(\mathrm{SO}_{2n+1}, \widetilde{\mathrm{Sp}}_n\right)$.
It is easily seen that the constant $\lambda_2$ which appears in the proof of
\cite[Lemma~4.3]{GT1} is given by $\displaystyle{\frac{\xi_F\left(1\right)}{ 2\,\xi_F\left(2\right)}}$
in our case.
Here $\xi_F\left(s\right)$ denotes the completed \emph{normalized} zeta function of $F$
given by 
\[
\xi_F\left(s\right) = |D_F|^{s \slash 2}\, \zeta_F\left(s\right),
\]
where $D_F$ is the discriminant of $F$, 
$\zeta_F$ is the completed Dedekind zeta function of $F$
defined by \eqref{e: dedekind} and we define $\xi_F\left(1\right) := \mathrm{Res}_{s=1} \,\xi_F\left(s\right)$ as in \cite[p.180]{GT1}.
Thus the constant $2$ does not appear in the Siegel-Weil formula above
unlike \cite[Proposition~5.3]{GT1}.
\end{proof}
We obtain the following Rallis inner product formula
from Proposition~\ref{SW} as in \cite{GT1}.
%
 \begin{theorem}\label{Rallis inner}
 For any non-zero decomposable vectors $\phi=\otimes_v\, \phi_v\in V_\pi$
 and $\varphi=\otimes_v\,\varphi_v\in 
 \mathcal{S}\left(\left(V \otimes Y_n^+\right)\left(\mA\right)\right)$,
 we have
 \begin{equation}\label{e: rallis inner product}
 \frac{\left(\theta_\psi^\varphi\left(\phi\right),\theta_\psi^\varphi\left(\phi\right)
 \right)}{\langle\phi,\phi\rangle}
 =C_G\cdot
 \frac{L\left(1/2,\pi\right)}{\prod_{j=1}^n\,\zeta_F\left(2j\right)}
 \cdot\prod_v Z_v^\circ\left(\phi_v, \varphi_v,\pi_v\right)
 \end{equation}
 where $Z_v^\circ\left(\phi_v, \varphi_v,\pi_v\right)=1$
 for almost all $v$.
\end{theorem}
%
\begin{proof}
Since $\sigma=\Theta_n\left(\pi,\psi\right)$ is cuspidal, by a similar computation as in
the proof of \cite[Proposition~6.1]{GT1},
it is shown that 
\begin{multline*}
\sum_i
\left( \theta_\psi^{\varphi_{1,i}}\left(\phi_1\right),\theta_\psi^{\varphi_{2, i}}\left(\phi_2\right) \right)=
\int_{(G \times G)(F) \backslash (G \times G)(\mA)} \phi_1(g_1) \overline{\phi_2(g_2)} \\
\cdot B_{-1}^{(2n+1, n)} \left(\sum_i \tau \left(\varphi_{1,i} \otimes \overline{\varphi_{2,i}} \right) \right)(\iota(g_1, g_2))\, dg_1 \, dg_2
\end{multline*}
for $\phi_i \in V_\pi$ and $\varphi = \sum_i \tau \left(\varphi_{1,i} 
\otimes \overline{\varphi_{2,i}}\right) \in 
\mathcal{S}^\circ\left(\left(\mathbb{W}^+ \otimes Y_n\right) \left(\mA\right)\right)$ 
such that $\Phi_\varphi = \otimes\, \Phi_v$ is factorizable,.
Then by Proposition~\ref{SW} and the doubling method, we obtain
\[
 \frac{\sum_i \left( \theta_\psi^{\varphi_{1,i}}\left(\phi\right),\theta_\psi^{\varphi_{2, i}}\left(\phi\right) \right)}{\langle\phi_1,\phi_2\rangle}
 =C_G\cdot
 \frac{L\left(1/2,\pi\right)}{\prod_{j=1}^n\,\zeta_F\left(2j\right)}
 \cdot\prod_v Z_v^\natural \left(0, \phi_{1, v}, \phi_{2, v}, \Phi_v,\pi_v\right)
\]
where
we define $Z_v^\natural \left(0, \phi_{1, v}, \phi_{2, v}, \Phi_v,\pi_v\right)$ by
\[
Z_v^\natural \left(0, \phi_{1, v}, \phi_{2, v}, \Phi_v,\pi_v\right)
=
\frac{\prod_{j=1}^n\zeta_{F_v}\left(2j\right)}{L\left(1\slash 2, \pi_v\right)}
\cdot \frac{1}{\langle\phi_{1, v},\phi_{2, v}\rangle_v}\cdot
Z_v\left(0, \phi_{1, v}, \phi_{2, v},\Phi_v,
\pi_v\right),
\]
in the same manner as \eqref{e: normalized doubling}.
Further, we may extend the  formula above to the whole space $\mathcal{S}\left(\left(V \otimes Y_n^+\right)\left(\mA\right)\right)$
by a simple argument as remarked in \cite[p.~243]{GT1}
and  \eqref{e: rallis inner product} holds.
\end{proof}
%
\begin{Remark}
 We note that there is a typo in
 the Rallis inner product formula stated in \cite[Theorem~6.6]{GT1}.
 It needs to be remedied as follows.
 There the Petersson inner product 
 of the theta lifts is essentially equal to
 $2$ times a certain  $L$-value.
 However the Siegel-Weil formula 
 \cite[Proposition~5.3]{GT1} implies that it is  essentially equal to
 $2^{-1}$ times the  $L$-value instead.
\end{Remark}
%
%
%
%
%
%
%
%
%
%
%
\subsection{Whittaker periods of
cusp forms on the metaplectic groups}
\label{ss: lapid mao}
We recall the Ichino-Ikeda type formula proved by  Lapid and Mao~\cite{LM}
for the Whittaker periods of cusp forms on the metaplectic groups.

Since $\widetilde{\mathrm{Sp}}_n$ splits trivially
over unipotent subgroups of $\mathrm{Sp}_n$ both locally and globally,
we regard these subgroups as subgroups of $\widetilde{\mathrm{Sp}}_n$.
For $a\in\gl_n$ and a symmetric $n\times n$ matrix $S$,
we denote by  $m\left(a\right)$ and $v\left(S\right)$
the elements of $\symp_n$ given by
\[
m\left(a\right)=\begin{pmatrix}a&0\\0&{}^ta^{-1}\end{pmatrix},
\quad
v\left(S\right)=\begin{pmatrix}1_n&S\\0&1_n\end{pmatrix}
\]
respectively.
Let $U_{\mathrm{Sp}}=\left\{v\left(S\right): {}^tS=S\right\}$
and $U_n$ the group of upper unipotent matrices in $\gl_n$.
Then the standard maximal unipotent subgroup $N=N_n$ of $\symp_n$
is given by 
\begin{equation}\label{e: unipotent decomposition}
N=m\left(U_n\right)U_{\mathrm{Sp}}.
\end{equation}
We define a character $\psi_\lambda$ of $N\left(\mathbb A\right)$
by
\begin{equation}\label{e: psi_lambda}
\psi_\lambda\left(m\left(u\right)v\left(S\right)\right)
=\psi\left(u_{1,2}+\cdots +u_{n-1,n}+\frac{\lambda}{2}s_{n,n}\right)
\end{equation}
where $u_{i,j}$ denotes the $\left(i,j\right)$-entry of $u$
and $s_{n,n}$ the $\left(n,n\right)$-entry of $S$.
%
\begin{Definition}\label{d: whittaker period}
For an automorphic form $\tilde{\phi}$
on $\widetilde{\mathrm{Sp}}_n\left(\mA\right)$,
its 
$\psi_\lambda$-Whittaker period $W\left(\tilde{\phi};\psi_\lambda\right)$
is defined by
\begin{equation}\label{e: metaplectic Whittaker}
W\left(\tilde{\phi};\psi_\lambda\right):=
\int_{N\left(F\right)\backslash N\left(\mA\right)}\tilde{\phi}\left(n\right)
\psi_\lambda\left(n\right)^{-1}\, dn.
\end{equation}

An automorphic representation of 
$\widetilde{\mathrm{Sp}}_n\left(\mA\right)$ is called  \emph{$\psi_\lambda$-generic}
when  $W\left(\,\,\,;\psi_\lambda\right)$
does not vanish identically
on its space of automorphic forms.
\end{Definition}

As we noted \eqref{e: theta lift}, 
\begin{equation}\label{e: genericiy}
\text{$\sigma=\Theta_n\left(\pi,\psi\right)$
is $\psi_\lambda$-generic, irreducible and cuspidal.}
\end{equation}
Let $\sigma=\otimes_v\,\sigma_v$.
Then by Adams and Barbasch~\cite{AB},
\begin{equation}\label{e: discrete series}
\text{
$\sigma_v$ is a discrete series representation
at any archimedean place $v$}
\end{equation}
 since so is $\pi_v$.
Let $\pi^\circ$ be the theta lift of  $\sigma$ to $\mathbb{G} \left(\mA \right)$ with respect to $\psi^{-\lambda}$,
which is globally generic by \cite[Proposition~1, 3]{FM0}.
Let $\Sigma=\Pi\otimes\chi_E$ where $\Pi$ is a weak lift of $\pi$
to $\gl_{2n}\left(\mA\right)$.
Then by \cite[Lemma~1]{FM0} and its proof, we have
\[
L\left(s,\Sigma_v\right)=L\left(s,\pi_v^\circ \right)
=L\left(s,\pi_v\times\chi_{E,v}\right)
\]
at every place $v$.
Thus we may say that $\Sigma$ is a weak lift of $\pi^\circ$
to $\gl_{2n}\left(\mA\right)$.

At each place $v$, 
we choose  a $\widetilde{\mathrm{Sp}}_n\left(F_v\right)$-invariant
Hermitian inner product $\left(\,\, , \,\,\right)_v$  on $V_{\sigma_v}$
so that
we have $\left(\tilde{\varphi}_1,\tilde{\varphi}_2\right)=
\prod_v\left(\tilde{\varphi}_{1,v},\tilde{\varphi}_{2,v}\right)_v$
for any decomposable vectors $\tilde{\varphi}_1=\otimes_v
\,\tilde{\varphi}_{1,v}$, $\tilde{\varphi}_2=\otimes_v
\,\tilde{\varphi}_{2,v}\in V_\sigma$.

Then by Lapid and Mao~\cite[Corollary~1.4]{LM},
we have the following theorem.
%
\begin{theorem}\label{t: LM formula}
For any non-zero decomposable cusp form
 $\tilde{\varphi}=\otimes_v\,
\tilde{\varphi}_v\in V_\sigma$, we have
\begin{equation}\label{e: LM formula}
\frac{\left| W\left(\tilde{\varphi};\psi_\lambda\right)\right|^2}{
\left(\tilde{\varphi},\tilde{\varphi}\right)}
=
2^{-l}\cdot \frac{L\left(1\slash 2, \pi\times\chi_E\right)
\prod_{j=1}^n\zeta_F\left(2j\right)}{L\left(1,\pi,\mathrm{Ad}\right)}
\cdot\prod_v I_v\left(\tilde{\varphi}_v\right)
\end{equation}
where $I_v\left(\tilde{\varphi}_v\right)$ is the stable integral
defined by
\begin{equation}\label{e: LM local integral}
I_v\left(\tilde{\varphi}_v\right):=
\frac{L\left(1,\pi_v, \mathrm{Ad}\right)}{
L\left(1\slash 2,\pi_v\times\chi_{E,v}\right)
\prod_{j=1}^n\zeta_{F_v}\left(2j\right)}
\int_{N_v}^{\mathrm{st}}
\frac{\left(\sigma_v\left(n_v\right)\tilde{\varphi}_v,
\tilde{\varphi}_v\right)_v}{\left(\tilde{\varphi}_v,
\tilde{\varphi}_v\right)_v}
\cdot \psi_\lambda^{-1}\left(n_v\right)\, dn_v
\end{equation}
and we have 
\begin{equation}\label{e: almost all}
I_v\left(\tilde{\varphi}_v\right)=1
\quad\text{for almost all places $v$ of $F$.}
\end{equation}
\end{theorem}
%
\begin{Remark}\label{r: stability}
When $v$ is non-archimedean, 
the integrand of \eqref{e: LM local integral} does have a stable integral over $N_v$
by \cite[Proposition~2.3]{LM0}.
When $v$ is archimedean, by \eqref{e: discrete series},
the integrand of \eqref{e: LM local integral}
is integrable over $N_v$ as explained in \cite[p.~455]{LM0}.
Thus $I_v\left(\tilde{\varphi}_v\right)$ at an archimedean place $v$ is
indeed given by the absolutely convergent integral over $N_v$.
\end{Remark}
The assertion \eqref{e: almost all} 
was actually proved by Ginzburg, Rallis and 
Soudry~\cite{GRS11} prior to \cite{LM0}.
Also 
we recall that for  any place $v$, we have
\begin{equation}\label{e: metaplectic mult one}
\dim_{\mathbb C}\mathrm{Hom}_{N_v}
\left(\sigma_v,\psi_{\lambda,v}\right)\le 1
\end{equation}
by Szpruch~\cite[Theorem~3.1]{Sz}
and Liu and Sun~\cite[Theorem~A]{LS}
for the non-archimedean case and the archimedean case,
respectively.
%
%
%
%
%
%
%
%
%
%
%
%
\subsection{Pull-back of the $\psi_\lambda$-Whittaker period}
\label{sect;whittaker}
Since $B_{\lambda,\psi}\not\equiv 0$ on $V_\pi$,
we have
$\operatorname{Hom}_{R_{\lambda,v}}
\left(\pi_v,\chi_{\lambda, v}\right)\ne\left\{0\right\}$ for any place $v$
of $F$.
Hence when $v$ is non-archimedean, $\alpha_v\not\equiv 0$
by \eqref{e: multiplicity one}.
\emph{Here we proceed further assuming the statement (1)
of Theorem~\ref{t: main theorem},
i.e. $\alpha_v\not\equiv 0$ at \emph{any} place $v$ of $F$, which we shall
prove later in 
\ref{ss: local GGP real}.}
%

By the multiplicity one property \eqref{uniqueness}  of the special Bessel model,
there exists $C\in\mathbb C^\times$ such that
\begin{equation}\label{e: uniqueness of special Bessel}
B_{\lambda,\psi}\left(\phi\right)
\cdot\overline{B_{\lambda,\psi}\left(\phi^\prime\right)}
=C\cdot\prod_v\,\alpha_v^{\natural}\left(\phi_{v},\phi_{v}^\prime\right)
\end{equation}
for any non-zero decomposable cusp forms $\phi=\otimes_v\, \phi_{v}$,
$\phi^\prime=\otimes_{v}\, \phi_{v}^\prime\in V_\pi$.
Here we note that $\alpha_v^\natural\left(\phi_{v} ,\phi_{v}^\prime \right)=1$
for almost all $v$.
In particular when $\phi=\phi^\prime$, 
we have
\begin{equation}\label{e: uniqueness of special Bessel3}
B_{\lambda,\psi}\left(\phi\right)
\cdot\overline{B_{\lambda,\psi}\left(\phi\right)}
=C\cdot\prod_v\,\alpha_v^\natural\left(\phi_v,\phi_v\right).
\end{equation}
When $B_{\lambda,\psi}\left(\phi\right)=0$,
we have $\alpha_v^\natural\left(\phi_v,\phi_v\right)=0$ at some place $v$
by \eqref{e: uniqueness of special Bessel3}.
Hence both  sides of \eqref{e: main identity} vanish
and the equality \eqref{e: main identity} holds.
Thus from now on we suppose that $B_{\lambda,\psi}\left(\phi\right)\ne 0$.
Since $\alpha_v^\natural(\phi_v, \phi_v) \ne 0$ for any $v$ by 
\eqref{e: uniqueness of special Bessel3}, we may define
\begin{equation}\label{e: alpha-knot}
\alpha_v^\circ\left(g_v;\phi_v \right):
=\alpha_v^\natural\left(\pi_v\left(g_v\right)\phi_{v},\phi_{ v} \right)\slash \alpha_v^\natural\left(
\phi_v,\phi_v \right)\quad\text{for $g_v\in G_v$}
\end{equation}
 for every place $v$ of $F$.
Then from \eqref{e: uniqueness of special Bessel}, for $g=\left(g_v\right)\in G\left(\mA\right)$,
we have
\begin{align*}
B_{\lambda,\psi}\left(\pi\left(g\right)\phi\right)\cdot
\overline{B_{\lambda,\psi}\left(\phi\right)}&=
C\cdot\prod_{v}
\alpha_v^\natural\left(\pi_v\left(g_v\right)\phi_v,\phi_v\right)
\\
&=B_{\lambda,\psi}\left(\phi\right)\cdot\overline{
B_{\lambda,\psi}\left(\phi\right)}
\cdot\prod_v \alpha_v^\circ\left(g_v;\phi_v\right)
\end{align*}
where $\alpha_v^\circ\left(g_v;\phi_v\right)=1$  for almost all $v$.
Hence
\begin{equation}\label{e: uniqueness of special Bessel2}
B_{\lambda,\psi}\left(\pi\left(g\right)\phi\right)=
B_{\lambda,\psi}\left(\phi\right)\cdot
\prod_v\,\alpha_v^\circ\left(g_v;\phi_v\right).
\end{equation}
%

Now we recall the pull-back formula for the $\psi_\lambda$-Whittaker
period. 
We identify $V\otimes Y_n^+$ with $V^n$.
The group $G$ acts on $V^n$ from the left by
\begin{equation}
\label{G act V}
g \cdot x=\left(gv_1,\cdots ,gv_n\right)\quad
\text{for $x=\left(v_1,\cdots, v_n\right)\in V^n$.}
\end{equation}
Then
by \cite[(9)]{Fu}, for $\varphi\in\mathcal S\left(V\left(\mA\right)^n\right)$,
we have
\begin{equation}\label{e: pull-back0}
W\left(\theta_\psi^\varphi\left(\phi\right);\psi_\lambda\right)=
\int_{R_\lambda^\prime\left(\mA\right)\backslash G\left(\mA\right)}
\varphi\left(g^{-1} \cdot \left(e_{-1}, \cdots , e_{-n+1},e_\lambda\right) \right)
B_{\lambda,\psi}\left(\pi\left(g\right)\phi\right)\, dg
\end{equation}
where 
\begin{equation}\label{e: stabilizer}
R_\lambda^\prime:=
\left\{g\in G:
\text{$ge_{-j}=e_{-j}$ for $1\le j\le n-1$,
$ge_\lambda=e_\lambda$}
\right\}.
\end{equation}
The  integral~\eqref{e: pull-back0} is well-defined since $\chi_\lambda$ is trivial on $R_\lambda^\prime(\mA)$. 
Here we note that
\begin{equation}\label{e: stabilizer1}
R_\lambda^\prime=D_\lambda S^\prime_\lambda
\quad
\text{where 
$S^\prime_\lambda=\left\{
\begin{pmatrix}1_{n-1}&A&B\\0&1_3&A^\prime\\
0&0&1_{n-1}\end{pmatrix}\in S^\prime
: Ae_\lambda=0\right\}$.}
\end{equation}

For each $j$ $\left(1\le j\le n-1\right)$, let $L_j$ be the subspace of $L$ spanned
by $e_{-j}$, $e_\lambda$ and $e_j$.
Then for $a\in F$, let $s_j\left(a\right)$ denote the element of $G$ such that
\begin{equation}\label{e: s_j}
s_j\left(a\right)\mid_{L_j^\perp}=1_{L_j^\perp}
\quad\text{and}\quad
s_j\left(a\right)\mid_{L_j}=\begin{pmatrix}1&a&-\lambda^{-1}a^2/2\\
0&1&-\lambda^{-1}a\\0&0&1\end{pmatrix}
\end{equation}
with respect to the basis $\left\{e_{-j},e_\lambda, e_j\right\}$ of $L_j$.
Also for each $j$ $\left(1\le j\le n-1\right)$, let us define a subgroup $S_j^\prime$  
of $S^\prime$ by
\[
S_j^\prime:=
\left\{\begin{pmatrix}1_{n-1}&A&B\\0&1_3&A^\prime\\0&0&1_{n-1}\end{pmatrix}
: Ae_\lambda\in  Fe_{-1}+\cdots +Fe_{-j}
\right\}
\]
and $S_0^\prime:= S_\lambda^\prime$.
We recall that  $S^\prime$ has a filtration
\begin{equation}\label{e: filtration of S prime}
S_\lambda^\prime=S_0^\prime\lhd S_1^\prime\lhd \cdots \lhd S_{n-1}^\prime=S^\prime
\end{equation}
 and we have 
\begin{equation}\label{e: S quotient}
S_{j-1}^\prime\backslash S_{j}^\prime \simeq\left\{s_j\left(a\right): a\in F\right\}.
\end{equation}
We also  note the induced filtration of $R_\lambda$, namely
\[
R_\lambda^\prime = D_\lambda S_\lambda^\prime=
D_\lambda S_0^\prime\lhd D_\lambda S_1^\prime\lhd \cdots \lhd D_\lambda S_{n-1}^\prime = D_\lambda S^\prime \lhd D_\lambda S^\prime S^{\prime \prime}
 =R_\lambda.
\]
%

\emph{
Let  $\varphi\in\mathcal S\left(V\left(\mathbb A\right)^n\right)$ be of the form $\varphi=\otimes_v\,\varphi_v$
where $\varphi_v\in\mathcal S\left(V\left(F_v\right)^n\right)$
and suppose
that
the local integral 
\begin{equation}\label{e: lambda}
\mathcal{L}_v\left(\varphi_v;\phi_v\right):=
\int_{R^\prime_{\lambda,v}\backslash G_v}
\varphi_v\left(g_v^{-1} \cdot \left(e_{-1}, \cdots , e_{-n+1},e_\lambda\right) \right)
\alpha_v^\circ\left(g_v;\phi_v\right)
\, dg_v
\end{equation}
converges absolutely and $\mathcal{L}_v\left(\varphi_v;\phi_v\right)\ne 0$
at each place $v$.}
Then  $\mathcal{L}_v\left(\varphi_v;\phi_v\right)=1$ for almost all $v$
and we may rewirite
\eqref{e: pull-back0} as
\begin{equation}\label{e: pull-back1}
W\left(\theta_\psi^\varphi\left(\phi\right);\psi_\lambda\right)=
\left(C_G\,C_\lambda^{-1}\right)\cdot
B_{\lambda,\psi}\left(\phi\right)\cdot
\prod_v\, \mathcal{L}_v\left(\varphi_v;\phi_v\right)
\end{equation}
and we have
 $W\left(\theta_\psi^\varphi\left(\phi\right);\psi_\lambda\right)\ne 0$.

Let $\Theta\left(\pi_v,\psi_v\right):=\operatorname{Hom}_{G_v}
\left(\omega_{\psi_v},\bar{\pi}_v\right)$
where $\omega_{\psi_v}$ is the local Weil representation
of $G_v\times\widetilde{\mathrm{Sp}}_n\left(F_v\right)$
realized on $\mathcal S\left(V\left(F_v\right)^n\right)$,
the space of Schwartz-Bruhat functions on $V\left(F_v\right)^n$.
As in the global case (e.g. see \cite[p.94]{Fu}), the action of $G_v\times\widetilde{\mathrm{Sp}}_n\left(F_v\right)$
via $\omega_{\psi_v}$  on $\varphi\in \mathcal S\left(V\left(F_v\right)^n\right)$
is given by the following formulas:
\begin{subequations}
\label{weil action}
%
\begin{equation}\label{weil action-a}
\omega_{\psi_v}(g, 1) \varphi(x) = \varphi(g^{-1} \cdot x), \quad g \in G_v,
\end{equation}
%
\begin{equation}\label{weil action-b}
 \omega_{\psi_v}(1, (m(a), \varepsilon)) \varphi(x)
= \varepsilon\, \frac{\gamma_\psi(1)}{\gamma_\psi((\det a)^{2n+1})} |\det a|^{n+\frac{1}{2}} \varphi(xa),\quad a \in \mathrm{GL}_n(F_v),
\end{equation}
%
\begin{equation}\label{weil action-c}
 \omega_{\psi_v}(1, v(S)) \varphi(x)
= \psi_v \left(\frac{1}{2} \mathrm{tr} \left(\mathrm{Gr}(x)S \right) \right) \varphi(x),
\end{equation}
%
\end{subequations}
where $\gamma_\psi$ denotes the Weil constant and
$\mathrm{Gr}\left(x\right)$ denotes the Gram matrix
$\left(\left(x_i,x_j\right)\right)$ 
for
$x  = (x_1, \dots, x_n) \in V(F_v)^n$.
We recall that for $\sigma=\Theta_n\left(\pi,\psi\right)$,
we have $\sigma=\otimes_v\,\sigma_v$
where $\sigma_v=\theta\left(\pi_v,\psi_v\right)$,
the unique
irreducible quotient of $\Theta\left(\pi_v,\psi_v\right)$
determined by the Howe duality.
The Howe duality was proved by
 Howe~\cite{Ho1} at archimedean places, by Waldspurger~\cite{Wa}
 at odd non-archimedean places and finally by Gan and Takeda~\cite{GT}
 at all non-archimedean places, respectively.
Let 
\[
\theta_v:\mathcal S\left(V\left(F_v\right)^n\right)\otimes V_{\pi_v}\to
V_{\sigma_v}
\]
be the $G_v\times\widetilde{\mathrm{Sp}}_n\left(F_v\right)$-equivariant
linear map, which is unique up to multiplication by a scalar.
Since the mapping
\[
\mathcal S\left(V\left(\mA\right)^n\right)\otimes V_\pi
\ni\left(\varphi^\prime,\phi^\prime\right)\mapsto
\theta_\psi^{\varphi^\prime}\left(\phi^\prime\right)\in V_\sigma
\]
is $G_v\times\widetilde{\mathrm{Sp}}_n\left(F_v\right)$-equivariant  at any place $v$,
by the uniqueness of $\theta_v$, we may adjust $\left\{\theta_v\right\}_v$ so that 
\[
\theta_\psi^{\varphi^\prime}\left(\phi^\prime\right)=\otimes_v\,
\theta_v\left(\varphi_v^\prime\otimes\phi_v^\prime\right)
\quad\text{for $\phi^\prime=\otimes_v\,\phi_v^\prime\in V_\pi$
and $\varphi^\prime=\otimes_v\, \varphi_v^\prime
\in\mathcal S\left(V\left(\mA\right)^n\right)$.}
\]
%
%
%
Hence by combining \eqref{e: rallis inner product}, \eqref{e: LM formula}
and \eqref{e: pull-back1},
we have
\begin{multline}\label{e: preformula}
\frac{\left|B_{\lambda,\psi}\left(\phi\right)\right|^2}{
\langle\phi,\phi\rangle}=C_G^{-1}\,C_\lambda^2\cdot 2^{-l}
\cdot
\frac{L\left(1/2,\pi\right)\,L\left(1/2,\pi\times\chi_E\right)}{
L\left(1,\pi,\operatorname{Ad}\right)}
\\
\times
\prod_v
\,
\frac{Z_v^\circ\left(\phi_v,\varphi_v,\pi_v\right)\, I_v\left(
\theta_v\left(\varphi_v\otimes\phi_v\right)\right)}{
\left| \mathcal{L}_v\left(\varphi_v;\phi_v\right)\right|^2}
\end{multline}
where
\[
\frac{Z_v^\circ\left(\phi_v,\varphi_v,\pi_v\right)\, I_v\left(
\theta_v\left(\varphi_v\otimes\phi_v\right)\right)}{
\left| \mathcal{L}_v\left(\varphi_v;\phi_v\right)\right|^2}=1
\]
for almost all $v$.
%
%
%

Since the right hand side of \eqref{e: main identity}
does not depend on the decompositions of the global
Tamagawa measures 
\eqref{e: measure comparison}, we may 
take specific local measures
which are suitable for our further considerations on the local
integrals appearing on the right hand side of \eqref{e: preformula}.
In Section~\ref{s: local identity}, we shall specify local measures
and show that we have
\begin{equation}\label{e: explicit constant}
C_G^{-1}\, C_\lambda=
\frac{\prod_{j=1}^n\zeta_F\left(2j\right)}{L\left(1,\chi_E\right)}
\end{equation}
and the following proposition holds.
%
\begin{proposition}
\label{prp: local equality}
Let $v$ be an arbitrary place of $F$.
For a given $\phi_v \in V_{\pi_v}$ satisfying $\alpha_v\left(\phi_v,\phi_v\right)\ne 0$, there exists $\varphi_v \in \mathcal{S}(V\left(F_v\right)^n)$
such that the local integral $\mathcal{L}_v\left(\varphi_v;\phi_v\right)$
converges absolutely, $\mathcal{L}_v\left(\varphi_v;\phi_v\right) \ne 0$
 and the  equality
\begin{equation}\label{e: local equality}
\frac{Z_v^\circ\left(\phi_v,\varphi_v,\pi_v\right)\, I_v\left(
\theta_v\left(\varphi_v\otimes\phi_v\right)\right)}{
\left| \mathcal{L}_v\left(\varphi_v;\phi_v\right)\right|^2}
=\frac{\alpha_v^\natural\left(\phi_v,\phi_v\right)}{
\langle\phi_v,\phi_v\rangle_v}
\end{equation}
holds
 with respect to the specified local measures.
\end{proposition}
%
Then for  $\phi=\otimes_v\,\phi_v\in V_\pi$ such that $B_{\lambda,\psi}\left(\phi\right)
\ne 0$, 
it is clearly seen  that the main identity \eqref{e: main identity} holds
by taking $\varphi_v \in \mathcal{S}(V\left(F_v\right)^n)$
as in Proposition~\ref{prp: local equality} for each place $v$
and by combining \eqref{e: preformula}, \eqref{e: explicit constant}
and \eqref{e: local equality}.
%
%
%
%
%
%
%
%
\section{Proof of the local equality}\label{s: local identity}
%
%
%
%
\subsection{Specification of local measures}\label{ss: Haar measure constants}
Recall that the group $G$ acts on $V^n$ from the left 
 by \eqref{G act V}.
Let 
\begin{equation}\label{e: x_0}
x_0:=\left(e_{-1},\cdots, e_{-n+1},e_\lambda\right)\in V^n
\end{equation}
and $\mathcal X_\lambda:=G \cdot x_0\subset V^n$.
Since $R^\prime_\lambda$ defined by \eqref{e: stabilizer} is the stabilizer of $x_0$, 
$R_\lambda^\prime\backslash G\ni g\mapsto g^{-1}\cdot x_0\in \mathcal X_\lambda$
is a $G$-homogeneous space isomorphism with the right action of $G$ on $\mathcal X_\lambda$
given by $\mathcal X_\lambda\ni x\mapsto g^{-1}\cdot x\in \mathcal X_\lambda$.
We note that $\mathcal X_\lambda$ is a locally closed subvariety of  $V^n$
since $\mathcal X_\lambda$ is a set of $x=\left(x_1,x_2,\cdots, x_n\right)\in V^n$
such that $\mathrm{Gr}\left(x\right)=\mathrm{Gr}\left(x_0\right)$
and $x_1,x_2,\cdots , x_n$ are linearly independent,
 by Witt's theorem.

Let $\omega$ and $\omega_G$
be non-zero gauge forms on $V^n$ and $G$, respectively.
Let $\omega_0$ be the gauge form on $\mathcal X_\lambda$
given by pulling back $\omega$ via the inclusion $\mathcal{X}_\lambda \hookrightarrow V^n$.
We choose a gauge form $\omega_\lambda$ on $R^\prime_\lambda$
such that $\omega_G$, $\omega_0$ and $\omega_\lambda$
match algebraically in the sense of Weil~\cite[p.24]{Weil},
i.e. $\omega_G=\omega_0\,\omega_\lambda$.
Also, we denote by $\omega_D$ the gauge form given by pulling back $\omega_\lambda$ via $D \hookrightarrow R_\lambda^\prime$.

In \cite{Gr}, Gross associated a motive of Artin-Tate type to a connected reductive
algebraic group over $F$.
Thus let $M_G$ be the motive associated to $G$ and $M_G^\vee\left(1\right)$ its twisted
dual motive.
The local Tamagawa measure $dg_v$ on $G_v$ corresponding to $\omega_G$
is given by $dg_v = L_v(M_{G}^\vee(1)) \cdot \left|\omega_G\right|_v$
at each place $v$ of $F$.
We refer to Gross~\cite{Gr} and Rogawski~\cite[1.7]{Rogawski} 
for the details concerning the definition of local Tamagawa measures.
Then the Haar measure constant $C_G$ defined by 
$dg=C_G\prod_v \,dg_v$,
where $dg$ is the Tamagawa measure, is given by
\begin{equation}\label{e: Haar measure constant 1}
C_G=\left(\prod_{j=1}^n\,\zeta_F\left(2j\right)\right)^{-1}.
\end{equation}

Similarly we specify the measure $dt_v$ on $D_{\lambda,v}$
to be the local Tamagawa measure corresponding to $\omega_D$
at each place $v$.
Then the Haar measure constant $C_\lambda$ defined by
$dt=C_\lambda\prod_v\, dt_v$, where
$dt$ is the Tamagawa measure, is given by
\begin{equation}\label{e: Haar measure constant 2}
C_\lambda=\frac{1}{L\left(1,\chi_E\right)}.
\end{equation}

For the unipotent group $S^\prime_{\lambda}$,
let the measure $ds^\prime_v$ on $S^\prime_{\lambda,v}$ at each place $v$
to be the measure specified in \ref{ss: notation} for unipotent groups.
We define the measure $dr^\prime_v$ on $R^\prime_{\lambda, v}=
D_{\lambda,v}S^\prime_{\lambda, v}$ by $dr^\prime_v=dt_v\, ds^\prime_{\lambda,v}$.
Finally we take the quotient measure $dh_v$
on $R_{\lambda, v}^\prime \backslash G_v$ so that 
\begin{equation}\label{e: quotient measure}
dg_v=dh_v\,dr^\prime_v
\end{equation}
holds.
It is clear from \eqref{e: Haar measure constant 1} and
\eqref{e: Haar measure constant 2} that 
the equality \eqref{e: explicit constant} holds
with these choices of the local measures.
%
%
%
%
%
%
%
%
%
%
%
%
%
%
%
%
%
%
%
%
%
%
%
%
%
\subsection{Two sesquilinear forms on $V_\sigma$}
\emph{Since our consideration is purely local from now
on till the end of \ref{ss: proposition 4}, 
we suppress the subscript $v$ expressing a place from the notation,
e.g. $F$ now denotes a local field.}

We construct two sesquilinear forms on $V_\sigma$ which satisfy the same
transformation property with respect to the subgroup $N$
of $\widetilde{\symp}_n\left(F\right)$.
%
\subsubsection{Sesquilinear form $\mathcal W$}\label{sss: W}
First we define a Hermitian inner product 
$\mathcal B_\omega$ on $\mathcal S\left(V^n\right)$ by
%
\begin{equation}\label{e: weil hermitian1}
\mathcal B_\omega\left(\varphi,\varphi^\prime\right):=
\int_{V^n}
\varphi\left(x\right)\,\overline{\varphi^\prime\left(x\right)}\,
dx
\quad\text{for $\varphi,\varphi^\prime\in \mathcal S\left(V^n\right)$}
\end{equation}
where $dx$ denotes the measure
corresponding to the gauge form $\omega$ on $V^n$
in \ref{ss: Haar measure constants}.
Then Liu~\cite[Lemma~3.19]{Liu} proved that the integral
\begin{equation}\label{e: weil hermitian2}
Z^\flat\left(\phi, \phi^\prime; \varphi, \varphi^\prime\right) 
= \int_{G} 
\langle\pi\left(g\right)\phi,\phi^\prime\rangle\,
\mathcal B_\omega\left(\omega_\psi\left(g\right)\varphi,\varphi^\prime\right)
\, dg
\end{equation}
converges absolutely for $\phi,\phi^\prime\in V_\pi$ and 
$\varphi,\varphi^\prime\in\mathcal S\left(V^n\right)$.
We note that our setting belongs to Case 2 
in the proof of \cite[Lemma~3.19]{Liu}.

As in Gan and Ichino~\cite[16.5]{GI}, 
there exists uniquely
an $\widetilde{\symp}_n\left(F\right)$-invariant
Hermitian inner product  
$\mathcal B_\sigma:V_\sigma\times
V_\sigma\to\mathbb C$ satisfying
\[
\mathcal B_\sigma\left(\theta\left(\varphi\otimes\phi\right),
\theta\left(\varphi^\prime\otimes\phi^\prime\right)\right):=Z^\flat\left(\phi, \phi^\prime; \varphi, \varphi^\prime\right)
\]
for $\phi,\phi^\prime\in V_\pi$ and 
$\varphi,\varphi^\prime\in\mathcal S\left(V^n\right)$.
Here we note that for $h\in\widetilde{\symp}_n\left(F\right)$ we have
\begin{equation}\label{e: weil hermitian3}
\mathcal B_\sigma\left(\sigma\left(h\right)\theta\left(\varphi\otimes\phi\right),
\theta\left(\varphi^\prime\otimes\phi^\prime\right)\right)=
\mathcal B_\sigma\left(\theta\left(\omega_\psi\left(h\right)\varphi\otimes\phi\right),
\theta\left(\varphi^\prime\otimes\phi^\prime\right)\right).
\end{equation}
%
\begin{Definition}\label{d: sesquilinear 1}
We define a sesquilinear form $\mathcal W=\mathcal W_{\varphi,\varphi^\prime}:V_\sigma\times V_\sigma\to\mathbb C$ by
\begin{equation}\label{e: weil hermitian4}
\mathcal W\left(\tilde{\phi}_1,\tilde{\phi}_2\right):=
\int_{N}^{\mathrm{st}}
\mathcal B_\sigma\left(\sigma\left(n\right)\tilde{\phi}_1,
\tilde{\phi}_2\right)\,\psi_\lambda\left(n\right)^{-1}\,dn
\end{equation}
for $\tilde{\phi}_1,\tilde{\phi}_2\in V_\sigma$.
\end{Definition}
%
Recall that, by Remark~\ref{r: stability}, the integrand of
\eqref{e: weil hermitian4} has a stable integral over $N$ when $F$ is non-archimedean
and it is integrable when $F$ is archimedean. 

We note that for $n_1,n_2\in N$ and $\tilde{\phi}_1,\tilde{\phi}_2\in V_\sigma$, we have
\begin{equation}\label{e: weil hermitian5}
\mathcal W\left(\sigma\left(n_1\right)\tilde{\phi}_1,
\sigma\left(n_2\right)\tilde{\phi}_2\right)=
\psi_\lambda\left(n_1\right)\psi_\lambda\left(n_2\right)^{-1}
\cdot
\mathcal W\left(\tilde{\phi}_1,\tilde{\phi}_2\right).
\end{equation}
%
\subsubsection{Sesquilinear form $\mathcal W^\circ$}
For $\phi,\phi^\prime\in V_\pi$ and 
$\varphi \in C_c^\infty \left(V^n \right)$,
let
\begin{equation}\label{e: weil hermitian6}
\mathcal V\left(\phi,\phi^\prime;\varphi\right):=
\int_{R_\lambda^\prime\backslash G}
\left(\omega_\psi\left(g,1\right)\varphi\right)\left(x_0\right)
\cdot \alpha\left(\pi\left(g\right)\phi,\phi^\prime\right)\,
dg.
\end{equation}
Recall that 
\[
\alpha\left(\phi,\phi^\prime\right)
=
\int_{D_{\lambda}}\int_{S}^{\mathrm{st}}
\langle\pi\left(st\right)\phi,\phi^\prime\rangle\,
\chi_\lambda\left(s\right)^{-1}\, ds\,dt.
\]
For $\varphi \in C_c^\infty \left(V^n \right)$, the support of $R_\lambda^\prime \backslash G \ni g \mapsto \varphi(g^{-1} \cdot x_0)$ is compact
since $\mathcal{X}_\lambda$ is locally closed
 in $V^n$.
 Hence  the  integral \eqref{e: weil hermitian6} indeed converges absolutely.
 
 We note that when $\alpha\left(\phi,\phi\right)\ne 0$, we have
 \begin{equation}\label{e: V and L}
 \mathcal V\left(\phi,\phi;\varphi\right)=
 \alpha\left(\phi,\phi\right)\cdot \mathcal L\left(\varphi;\phi\right).
 \end{equation}
 Recall that $\mathcal L\left(\varphi;\phi\right)$ is defined by
 \eqref{e: lambda}.
 We also note that
\begin{equation}\label{e: weil hermitian7}
\mathcal V\left(\pi\left(g\right)\phi,\phi^\prime;\omega_\psi\left(g,1\right)\varphi\right)=
\mathcal V\left(\phi,\phi^\prime;\varphi\right)
\quad\text{for $g\in G$.}
\end{equation}
By \eqref{e: weil hermitian7}  there exists uniquely 
a linear form $\ell_{\phi^\prime,\varphi}:V_\sigma\to\mathbb C$
such that
\begin{equation}\label{e: weil hermitian8}
\ell_{\phi^\prime,\varphi}\left(\theta\left(\varphi\otimes\phi\right)\right)=
\mathcal V\left(\phi,\phi^\prime;\varphi\right)
\quad\text{for $\phi\in V_\pi$ and 
$\varphi \in C_c^\infty \left(V^n \right)$.}
\end{equation}
Then,  for $n\in N$ and $\tilde{\phi}\in V_\sigma$, we have
\begin{equation}\label{e: weil hermitian9}
\ell_{\phi^\prime,\varphi}\left(\sigma\left(n\right)\tilde{\phi}\right)=
\psi_\lambda\left(n\right)
\ell_{\phi^\prime,\varphi}\left(\tilde{\phi}\right).
\end{equation}
%
%
\begin{Definition}\label{d: another W}
For $\phi,\phi^\prime\in V_\pi$ and $\varphi,\varphi^\prime\in C_c^\infty \left(V^n \right)$,
we define a sesquilinear form $\mathcal W^\circ=\mathcal W^\circ_{\phi,\phi^\prime,\varphi,\varphi^\prime}:V_\sigma\times V_\sigma\to\mathbb C$
by
\begin{equation}\label{e: another W}
\mathcal W^\circ\left(\tilde{\phi}_1,\tilde{\phi}_2\right):=
\ell_{\phi^\prime,\varphi}\left(\tilde{\phi}_1\right)\cdot
\overline{\ell_{\phi, \varphi^\prime}\left(\tilde{\phi}_2\right)}
\end{equation}
for $\tilde{\phi}_1,\tilde{\phi}_2\in V_\sigma$.
\end{Definition}
%
It is clear from \eqref{e: weil hermitian9} that
for $n_1,n_2\in N$ and $\tilde{\phi}_1,\tilde{\phi}_2\in V_\sigma$,
\begin{equation}\label{e: property of another W}
\mathcal W^\circ\left(\sigma\left(n_1\right)\tilde{\phi}_1,
\sigma\left(n_2\right)\tilde{\phi}_2\right)=
\psi_\lambda\left(n_1\right)\psi_\lambda\left(n_2\right)^{-1}
\cdot
\mathcal W^\circ\left(\tilde{\phi}_1,\tilde{\phi}_2\right).
\end{equation}
%
\subsubsection{Comparison between $\mathcal W$ and $\mathcal W^\circ$}
%
First we note the following lemma whose proof is clear since
$\mathcal X_\lambda$
is locally closed in $V^n$.
\begin{lemma}\label{l: non-vanishing}
Suppose that $\alpha\left(\phi,\phi^\prime\right)\ne 0$.
Then for any open neighborhood $O_{x_0}$ of $x_0$
in $V^n$, there exists $\varphi\in C_c^\infty \left(V^n \right)$ such that 
$\mathrm{Supp}\left(\varphi\right)$, the support of $\varphi$,
is contained in 
$O_{x_0}$ and
$\mathcal V\left(\phi,\phi^\prime;\varphi
\right)\ne 0$.
In particular the linear form $\ell_{\phi^\prime,\varphi}$ on $V_\sigma$
defined by \eqref{e: weil hermitian8} is non-zero for such $\varphi$.
\end{lemma}
%
By the uniqueness of the $\psi_\lambda$-Whittaker model \eqref{e: metaplectic mult one},
the equalities \eqref{e: weil hermitian5} and \eqref{e: property of another W} 
imply that $\mathcal W$ is a scalar multiple of $\mathcal W^\circ$ when 
$\mathcal W^\circ$ is non-zero.
The following proposition states that the constant of proportionality is given explicitly.
%
\begin{proposition}\label{p: local equality}
Suppose that 
$\phi,\phi^\prime\in V_\pi$ satisfy
$\alpha\left(\phi,\phi^\prime\right)\ne 0$.

Then for any $\varphi,\varphi^\prime\in C_c^\infty\left(V^n\right)$  
satisfying $\ell_{\phi^\prime,\varphi}\left(\phi^\prime\right)\ne 0$
and $\ell_{\phi,\varphi^\prime}\left(\phi\right)\ne 0$, we have
\begin{equation}\label{e: proportionality}
\mathcal W_{\varphi,\varphi^\prime}=
\frac{C_{E\slash F}}{\alpha\left(\phi,\phi^\prime\right)}\cdot
\mathcal W^\circ_{\phi,\phi^\prime,\varphi,\varphi^\prime}
\end{equation}
where 
\begin{equation}\label{e: constant of prop}
C_{E\slash F}=\frac{L\left(1,\chi_{E\slash F}\right)}{\prod_{j=1}^n\,\zeta_F\left(2j\right)}.
\end{equation}
\end{proposition}
%
Let us  show that Proposition~\ref{prp: local equality} follows from
Proposition~\ref{p: local equality}
before proceeding to a proof of Proposition~\ref{p: local equality}.
%
%
%
\begin{proof}[Proof of Proposition~\ref{prp: local equality}]
Suppose that $\alpha(\phi, \phi) \ne 0$.
By Lemma~\ref{l: non-vanishing}, we may
 take $\varphi \in C_c^\infty(V^n)$ so that
$\mathcal V\left(\phi,\phi;\varphi\right)=\alpha\left(\phi,\phi\right)
\cdot\mathcal L\left(\varphi;\phi\right) \ne 0$. 
Then $\mathcal W^\circ=\mathcal W^\circ_{\phi,\phi,\varphi,\varphi}$ is non-zero.
Hence by \eqref{e: proportionality}, we have
\[
\mathcal W\left(\theta\left(\varphi\otimes\phi\right),\theta\left(\varphi\otimes\phi\right)\right)
=\frac{C_{E\slash F}}{\alpha\left(\phi,\phi\right)}\cdot
\mathcal W^\circ\left(\theta\left(\varphi\otimes\phi\right),\theta\left(\varphi\otimes\phi\right)\right),
\]
i.e.
\begin{equation}\label{e: weil hermitian11}
\mathcal W\left(\theta\left(\varphi\otimes\phi\right),
\theta\left(\varphi\otimes\phi\right)\right)
=C_{E \slash F}\cdot \alpha\left(\phi,\phi\right)\cdot
\left| \mathcal L\left(\varphi;\phi\right)\right|^2
\end{equation}
by \eqref{e: V and L} and \eqref{e: weil hermitian8}.

On the other hand, by using 
the $\widetilde{\mathrm{Sp}}_n\left(F\right)$-invariant Hermitian inner product
$\mathcal{B}_\sigma(\,, \,)$ in the definition \eqref{e: LM local integral}
for $I\left(\theta(\varphi\otimes \phi)\right)$, we have
\begin{multline}\label{rem eq 1}
\mathcal B_\sigma\left(\theta\left(\varphi\otimes\phi\right), 
\theta\left(\varphi\otimes\phi\right)
\right) \cdot I\left(\theta(\varphi\otimes \phi)\right)
\\
=
\frac{L\left(1,\pi, \mathrm{Ad}\right)}{
L\left(1\slash 2,\pi \times\chi_{E}\right)
\prod_{j=1}^n\zeta_{F}\left(2j\right)}
 \cdot 
\mathcal W\left(\theta\left(\varphi\otimes\phi\right),
\theta\left(\varphi\otimes\phi\right)\right)
\end{multline}
by \eqref{e: weil hermitian4}.
Here by Gan and Ichino~\cite[16.3]{GI}, we have
\begin{equation}\label{e: one of Gan-Ichino}
\mathcal B_\sigma\left(\theta\left(\varphi\otimes\phi\right), 
\theta\left(\varphi\otimes\phi\right)\right)
=
Z\left(0, \phi,\phi,\Phi_{\tau\left(\varphi\otimes\varphi\right)}, \pi \right)
\end{equation}
where the right hand side is the local doubling integral
defined by \eqref{e: local doubling}.
Hence by rewriting \eqref{e: one of Gan-Ichino} in terms of 
$Z^\circ\left(\phi, \varphi, \pi\right)$ defined by \eqref{e: normalized doubling}, we have
\begin{equation}
\label{rem eq 2}
 \frac{\prod_{j=1}^n\zeta_{F}\left(2j\right)}{L\left(1\slash 2, \pi\right)} 
\cdot 
\mathcal B_\sigma\left(\theta\left(\varphi\otimes\phi\right), 
\theta\left(\varphi\otimes \phi\right)\right)
=
\langle \phi, \phi \rangle \cdot Z^\circ\left(\phi, \varphi, \pi\right).
\end{equation}
Thus by combining  \eqref{rem eq 1} and \eqref{rem eq 2}, we have
\begin{multline}\label{e: W-I}
\frac{L\left(1,\pi, \mathrm{Ad}\right)}{
L\left(1\slash 2,\pi\times\chi_{E}\right)
L\left(1\slash 2, \pi\right)}
 \cdot 
\mathcal W\left(\theta\left(\varphi\otimes\phi\right),
\theta\left(\varphi\otimes\phi\right)\right) \\
=\langle \phi, \phi \rangle \cdot Z^\circ\left(\phi, \varphi, \pi\right) \cdot I\left(\theta\left(\varphi\otimes\phi\right)\right).
\end{multline}

Substituting \eqref{e: weil hermitian11} into \eqref{e: W-I} yields
\eqref{e: local equality} and thus Proposition~\ref{prp: local equality} holds.
\end{proof}
%

\subsection{Reduction to another local equality}
%
%
Here we shall observe that Proposition~\ref{p: local equality} follows
from a local equality \eqref{e: weil hermitian13} below.
%

Since 
\[
\mathcal W^\circ\left(\theta\left(\varphi\otimes\phi\right),
\theta\left(\varphi^\prime\otimes\phi^\prime\right)\right)=
\mathcal V\left(\phi,\phi^\prime;\varphi\right)\cdot
\overline{\mathcal V\left(\phi^\prime,\phi;\varphi^\prime\right)}\ne 0
\]
by \eqref{e: weil hermitian8} and \eqref{e: another W},
the equality \eqref{e: proportionality} follows from
\begin{equation}\label{e: weil hermitian10-1}
\mathcal W\left(
\theta\left(\varphi\otimes\phi
\right),
\theta\left(\varphi^\prime\otimes\phi^\prime\right)
\right)
=
\frac{C_{E \slash F} }{\alpha\left(\phi,\phi^\prime\right)}\cdot
\mathcal V\left(\phi,\phi^\prime;\varphi\right)
\cdot
\overline{\mathcal V\left(\phi^\prime,\phi;\varphi^\prime\right)}.
\end{equation}
Here
\begin{multline}\label{e: def of V-V}
\mathcal V\left(\phi,\phi^\prime;\varphi\right)
\cdot
\overline{\mathcal V\left(\phi^\prime,\phi;\varphi^\prime\right)}
=
\int_{R_\lambda^\prime\backslash G}
\int_{R_\lambda^\prime\backslash G}
\alpha\left(\pi\left(h\right)\phi,\phi^\prime\right)\,
\overline{\alpha\left(\pi\left(h^\prime\right)\phi^\prime,\phi\right)}
\\
\times
\left(\omega_\psi\left(h,1\right)\varphi\right)\left(x_0\right)\,
\overline{\left(\omega_\psi\left(h^\prime,1\right)\varphi^\prime\right)\left(x_0\right)}\,
dh\, dh^\prime.
\end{multline}
We observe that a sesquilinear form $\mathcal A$ on $V_\pi$ defined by
\[
\mathcal A\left(\phi_1,\phi_1^\prime\right):=
\alpha\left(\phi_1,\phi^\prime\right)\overline{\alpha\left(\phi_1^\prime,\phi\right)}
\]
satisfies 
\[
\mathcal A\left(\pi\left(r\right)\phi_1,\pi\left(r^\prime\right)\phi_1^\prime\right)=
\chi_\lambda\left(r\right)\,\chi_\lambda\left(r^\prime\right)^{-1}\cdot
\mathcal A\left(\phi_1,\phi_1^\prime\right)
\]
for $r,r^\prime\in R_\lambda$.
Hence  the uniqueness of the special Bessel model
\eqref{uniqueness} implies that
there exists a constant $c^\prime$ such that $\mathcal A=c^\prime\cdot\alpha$.
Since $\alpha\left(\phi,\phi^\prime\right)\ne 0$, we have
\[
c^\prime=\mathcal A\left(\phi,\phi^\prime\right)\slash 
\alpha\left(\phi,\phi^\prime\right)=\alpha\left(\phi,\phi^\prime\right).
\]
Hence in the integrand of \eqref{e: def of V-V}, we have
\[
\alpha\left(\pi\left(h\right)\phi,\phi^\prime\right)\,
\overline{\alpha\left(\pi\left(h^\prime\right)\phi^\prime,\phi\right)}=
\mathcal A\left(\pi\left(h\right)\phi,\pi\left(h^\prime\right)\phi^\prime\right)
=\alpha\left(\phi,\phi^\prime\right)
\alpha\left(\pi\left(h\right)\phi,\pi\left(h^\prime\right)\phi^\prime\right).
\]

Thus the equality \eqref{e: weil hermitian10-1} follows from the following 
proposition.
%
%
%
%
%
\begin{proposition}\label{l: main lemma}
For any $\phi,\phi^\prime\in V_\pi$ and any $\varphi,\varphi^\prime\in C_c^\infty\left(V^n\right)$,
we have
\begin{multline}\label{e: weil hermitian13}
\mathcal W\left(
\theta\left(\varphi\otimes\phi
\right),
\theta\left(\varphi^\prime\otimes\phi^\prime\right)
\right)
=
C_{E\slash F} 
\\
\times
\int_{R_\lambda^\prime\backslash G}
\int_{R_\lambda^\prime\backslash G}
\alpha\left(\pi\left(h\right)\phi,\pi\left(h^\prime\right)\phi^\prime\right)\,
\left(\omega_\psi\left(h,1\right)\varphi\right)\left(x_0\right)\,
\overline{\left(\omega_\psi\left(h^\prime,1\right)\varphi^\prime\right)\left(x_0\right)}\,
dh\, dh^\prime.
\end{multline}
\end{proposition}
%
%
%
%
\begin{Remark}
Note that  we shall show \eqref{e: weil hermitian13} in more generality 
than just necessary to prove \eqref{e: weil hermitian10-1}
because of its later use in the proof of Corollary~\ref{c: genericity}.
In particular,
we do not assume  $\alpha\left(\phi,\phi^\prime\right)\ne 0$
in Proposition~\ref{l: main lemma}.
\end{Remark}
%
\begin{Remark}\label{r: local pull-back}
The equality \eqref{e: weil hermitian13} may be naturally regarded as a
local pull-back formula for the $\psi_\lambda$-Whittaker pairing.
\end{Remark}
%
%
%
%
\subsection{Proof of Proposition~\ref{l: main lemma}}
\label{ss: proposition 4}
%
%
%
By the definition in \ref{sss: W}, we have
\begin{multline}\label{e: weil hermitian14}
\mathcal W\left(
\theta\left(\varphi\otimes\phi
\right),
\theta\left(\varphi^\prime\otimes\phi^\prime\right)
\right)
=\int_{U_n}^{\mathrm{st}}
\int_{U_{\mathrm{Sp}}}^{\mathrm{st}}\int_G\int_{V^n}
\\
\left(\omega_\psi\left(1,m\left(u\right)v\right)\varphi\right)\left(g^{-1} \cdot x\right)
\, \overline{\varphi^\prime\left(x\right)}
\, \langle\pi\left(g\right)\phi,\phi^\prime\rangle
\,\psi_{\lambda}\left(m\left(u\right)v\right)^{-1}\, dx\, dg\, dv\, du.
\end{multline}
Here we use the decomposition \eqref{e: unipotent decomposition} of  $N$.
We shall show \eqref{e: weil hermitian13} by modifying the right hand side of 
\eqref{e: weil hermitian14} in steps.
%
\subsubsection{Inner triple integral}
We shall take care of the inner triple integral of  \eqref{e: weil hermitian14} by adapting 
Liu's computations in \cite[3.5]{Liu} to our setting.

Let $V^n_\circ$ be a subset of $V^n$ consisting of 
$\left(v_1,\cdots , v_n\right)\in V^n$ such that
$v_1,\cdots , v_n$ are linearly independent
and the inner product $\left(v_n,v_n\right)\ne 0$.
Then $V^n_\circ$ is open in $V^n$ and
$\mathrm{Vol}\left(V^n\setminus V^n_\circ,dx\right)=0$.
%

Let $\mathrm{Sym}^n$ denote the set of $n\times n$ symmetric matrices with
entries in $F$ and
\[
\mathrm{Sym}^n_\circ:=\left\{
S=\left(s_{i,j}\right)\in \mathrm{Sym}^n\mid
s_{n,n}\ne 0\right\}.
\]
We consider a mapping $\mathrm{Gr}:V_\circ^n\to \mathrm{Sym}^n_\circ$
given by the Gram matrix $\mathrm{Gr}\left(x\right)$
for $x\in V^n_\circ$.
It is clear that $\mathrm{Gr}$ is surjective.
For each $S\in\mathrm{Sym}^n_\circ$, 
we fix  $x_S\in V^n_\circ$ such that $\mathrm{Gr}\left(x_S\right)=S$.
Then by Witt's theorem, the fiber $\mathrm{Gr}^{-1}\left(S\right)$
of $S$ is given by
\[
\mathrm{Gr}^{-1}\left(S\right)=\left\{g^{-1} \cdot x_S \mid g\in G\right\}.
\]
Let $R_S^\prime$ denotes the stabilizer of 
$x_S$
 in $G$.
Then we may identify $\mathrm{Gr}^{-1}\left(S\right)$
with $R_S^\prime\backslash G$ as $G$-homogeneous spaces.
We have the following integration formula.
%
\begin{lemma}\label{l: measure}
For each $S\in\mathrm{Sym}^n_\circ$, there exists a
Haar measure $dr_S^\prime$ on $R_S^\prime$ such that
\begin{equation}\label{e: integration formula}
\int_{V^n}\Phi\left(x\right)\, dx=
\int_{\mathrm{Sym}^n_\circ}\int_{R_S^\prime\backslash G}
\Phi\left(h^{-1} \cdot x_S\right)\, dh_S\, dS
\end{equation}
for any $\Phi\in L^1\left(V^n\right)$.
Here $dh_S$ denotes the quotient measure
$dr^\prime_S\backslash dg$ on $R_S^\prime\backslash G$.
\end{lemma}
\begin{proof}
Since $\mathrm{Vol}\left(V^n\setminus V^n_\circ,dx\right)=0$,
we have
\[
\int_{V^n}\Phi\left(x\right)\, dx=
\int_{V_\circ^n}\Phi\left(x\right)\, dx
\]
for $\Phi\in L^1\left(V^n\right)$.
Then the lemma readily follows from  the observation above.
\end{proof}
%
\begin{Remark}
Let  $S_\circ:= \mathrm{Gr}\left(x_0\right)
\in\mathrm{Sym}^n_\circ$ where $x_0$ is given by \eqref{e: x_0}.
We take $x_S$ to be $x_0$
when $S=S_0$.
Then
 $R_{S_\circ}^\prime= R_\lambda^\prime$ 
 defined by  \eqref{e: stabilizer}, which has a decomposition
 $R_\lambda^\prime = D_\lambda S^\prime_\lambda$
as \eqref{e: stabilizer1}.
Recall that we take the local Tamagawa measure  
as explained in \ref{ss: Haar measure constants}. Recall also that the measure on the 
unipotent group $S_\lambda^\prime$ is taken as explained in  \ref{ss: notation}.
On the other hand, the quotient measure $dh_{S_\circ}$ on $R_\lambda^\prime \backslash G$ used in \eqref{e: integration formula} is the 
quotient measure of 
the local measures corresponding to the gauge forms $\omega_\lambda$ and $\omega_G$,
which are not normalized as local Tamagawa measures by 
the
 local $L$-factors. 
Hence the relationship between
the two quotient measures on $R_\lambda^\prime\backslash G$,
$dh_{S_\circ}$  in \eqref{e: integration formula} 
and $dh$ defined by \eqref{e: quotient measure},
is given by
\begin{equation}\label{e: two measures}
dh_{S_\circ} =  C_{E \slash F} \cdot  dh
\end{equation}
where $C_{E\slash F}$ is as \eqref{e: proportionality}.
\end{Remark}
%

Before proceeding further, we note the following lemma,
which is proved by an argument
similar to the one for \cite[Lemma~3.20]{Liu}
when $F$ is non-archimedean
and to the one for \cite[Proposition~3.22]{Liu} when $F$ is archimedean, respectively.
%
\begin{lemma}\label{l: integral formula1}
For $\varphi_1,\varphi_2\in C_c^\infty\left(V^n\right)$
and $\phi_1,\phi_2\in V_\pi$, let 
\[
\mathcal G_{\varphi_1,\varphi_2,\phi_1,\phi_2}
\left(S\right)=\int_G
\int_{R_S^\prime\backslash G}
\varphi_1\left((h g^\prime)^{-1} \cdot x_S\right)\,
\varphi_2\left(h^{-1} \cdot x_S\right)\,
\langle\pi\left(g^\prime\right)\phi_1,\phi_2\rangle\,
dh\, dg^\prime
\]
for $S\in \mathrm{Sym}_\circ^n$.
\begin{enumerate}
\item
When $F$ is non-archimedean,
the integral is absolutely convergent and is 
locally constant.
\item
When $F$ is archimedean,
the integral is absolutely convergent and 
is a function in $L^1\left(\mathrm{Sym}^n\right)$
which is continuous on $\mathrm{Sym}^n_\circ$.
\end{enumerate}
\end{lemma}
%

Now for $\varphi,\varphi^\prime\in C_c^\infty\left(V^n\right)$
and $\phi,\phi^\prime\in V_\pi$, let
\begin{equation}\label{e: f}
f_{\varphi, \varphi^\prime,\phi,\phi^\prime}(n)
:=
\int_G\int_{V^n}
\\
\left(\omega_\psi\left(1,n\right)\varphi\right)\left(g^{-1} \cdot x \right)
\, \overline{\varphi^\prime\left(x\right)}
\, \langle\pi\left(g\right)\phi,\phi^\prime\rangle \, dx\, dg
\end{equation}
for $n\in N$.
Then
\begin{equation}\label{e: first step}
\mathcal W\left(
\theta\left(\varphi\otimes\phi
\right),
\theta\left(\varphi^\prime\otimes\phi^\prime\right)
\right)
=\int_{U_n}^{\mathrm{st}}
\int_{U_{\mathrm{Sp}}}^{\mathrm{st}}
f_{\varphi,\varphi^\prime}\left(m\left(u\right)v\right)
\psi_\lambda\left(m\left(u\right)v\right)^{-1}\, dv\, du.
\end{equation}

Since 
\[
U_{\mathrm{Sp}}=\left\{v\left(S\right)=\begin{pmatrix}1_n&S\\0&1_n\end{pmatrix}
: S\in\mathrm{Sym}^n\right\},
\]
we may regard 
$\displaystyle{\int_{U_{\mathrm{Sp}}}^{\mathrm{st}}}$ as 
$\displaystyle{\int_{\mathrm{Sym}^n}^{\mathrm{st}}}$.
Then by rewriting the integration over $V^n$ in \eqref{e: f}
using the integration formula \eqref{e: integration formula},
we have the following lemma.
%
\begin{lemma}\label{l: integral formula2}
We have
\begin{multline}
\label{lem LS}
\int_{U_{\mathrm{Sp}}}^{\mathrm{st}} f_{\varphi, \varphi^\prime}(v)
\psi_{\lambda}\left(v\right)^{-1} \, dv
\\
=
C_{E \slash F} \cdot  \int_{G}\int_{R_\lambda^\prime\backslash G}
\left(\omega_\psi\left(hg,
1\right)\varphi\right)\left(x_0\right)
\, \overline{\varphi^\prime\left(h^{-1} \cdot x_0 \right)}
\, \langle\pi\left(g\right)\phi,\phi^\prime\rangle \, dh\, dg.
\end{multline}
\end{lemma}
%
\begin{proof}
The argument using the Fourier inversion
for the proof of \cite[Proposition~3.21]{Liu}
in the non-archimedean case and the one for
\cite[Corollary~3.23]{Liu} in the archimedean case
work mutatis mutandis, 
since Lemma~\ref{l: integral formula1} holds.
Thus we obtain \eqref{lem LS} 
by taking into account \eqref{e: two measures} also.
\end{proof}
%
By Lemma~\ref{l: integral formula2}, we have
\begin{multline}\label{e: weil hermitian15}
\mathcal W\left(
\theta\left(\varphi\otimes\phi
\right),
\theta\left(\varphi^\prime\otimes\phi^\prime\right)
\right)
=C_{E \slash F} \cdot \int_{U_n}^{\mathrm{st}}\int_{G}\int_{R_\lambda^\prime\backslash G}
\\
\left(\omega_\psi\left(hg,
m\left(u\right)\right)\varphi\right)\left(x_0\right)
\, \overline{\varphi^\prime\left(h^{-1} \cdot x_0\right)}
\, \langle\pi\left(g\right)\phi,\phi^\prime\rangle
\,\psi_{\lambda}\left(m\left(u\right)\right)^{-1}\, dh\, dg\,du.
\end{multline}
Then by a change of variable $g\mapsto h^{-1}g$
and  also noting  that $\langle\pi\left(h^{-1}g\right)\phi,\phi^\prime\rangle=
\langle\pi\left(g\right)\phi,\pi\left(h\right)\phi^\prime\rangle$, we may
write \eqref{e: weil hermitian15} as
\begin{multline}\label{e: weil hermitian16}
\mathcal W\left(
\theta\left(\varphi\otimes\phi
\right),
\theta\left(\varphi^\prime\otimes\phi^\prime\right)
\right)
=C_{E \slash F} \cdot\int_{U_n}^{\mathrm{st}}\int_{G}\int_{R_\lambda^\prime\backslash G}
\\
\left(\omega_\psi\left(g,
m\left(u\right)\right)\varphi\right)\left(x_0\right)
\, \overline{\varphi^\prime\left(h^{-1} \cdot x_0\right)}
\, \langle\pi\left(g\right)\phi,\pi\left(h\right)\phi^\prime\rangle
\,\psi_{\lambda}\left(m\left(u\right)\right)^{-1}\, dh\, dg\,du.
\end{multline}
%
Here the inner double integral on the right hand side of
\eqref{e: weil hermitian16}
 converges absolutely by Lemma~\ref{l: integral formula1}.
 Hence we may change the order of integration
 and we have
\begin{multline*}
\mathcal W\left(
\theta\left(\varphi\otimes\phi
\right),
\theta\left(\varphi^\prime\otimes\phi^\prime\right)
\right)
=C_{E \slash F} \cdot\int_{U_n}^{\mathrm{st}}\int_{R_\lambda^\prime\backslash G}\int_{G}
\\
\left(\omega_\psi\left(g,
m\left(u\right)\right)\varphi\right)\left(x_0\right)
\, \overline{\varphi^\prime\left(h^{-1} \cdot x_0\right)}
\, \langle\pi\left(g\right)\phi,\pi\left(h\right)\phi^\prime\rangle
\,\psi_{\lambda}\left(m\left(u\right)\right)^{-1}\, dg\, dh\,du.
\end{multline*}
%
Moreover, since the 
inner-most
integral converges absolutely, we may telescope  
the $G$-integration and we have
\begin{multline}\label{e: weil hermitian17}
\mathcal W\left(
\theta\left(\varphi\otimes\phi
\right),
\theta\left(\varphi^\prime\otimes\phi^\prime\right)
\right)
=C_{E \slash F} \cdot\int_{U_n}^{\mathrm{st}}\int_{R_\lambda^\prime
\backslash G}\int_{R_\lambda^\prime\backslash G}\int_{R_\lambda^\prime}
\\
\left(\omega_\psi\left(g,
m\left(u\right)\right)\varphi\right)\left(x_0\right)
\, \overline{\varphi^\prime\left(h^{-1} \cdot x_0\right)}
\, \langle\pi\left(r^\prime g\right)\phi,\pi\left(h\right)\phi^\prime\rangle
\,\psi_{\lambda}\left(m\left(u\right)\right)^{-1}\, dr^\prime\,
dg\, dh \,du.
\end{multline}
%
\begin{Remark}
\label{rem; abs conv}
As we have seen, because of Lemma~\ref{l: integral formula1},
\begin{equation}\label{e: r_lambda^prime}
\int_{R_\lambda^\prime}
\langle\pi\left(r^\prime g \right)\phi, \pi(h)\phi^\prime\rangle\,
dr^\prime,
\end{equation}
the most inner integral of \eqref{e: weil hermitian17},
converges absolutely.
This $R_\lambda^\prime$-integration  appears as an inner integral of the definition \eqref{e: local integral 1} for
$\alpha \left(\pi(g)\phi, \pi(h)\phi^\prime \right)$
since $R_\lambda^\prime=D_\lambda S_\lambda^\prime\subset R_\lambda$
and $\chi_\lambda\left(r^\prime\right)=1$ for $r^\prime\in R_\lambda^\prime$.
\end{Remark}
%
\subsubsection{Stable integration over $U_n$}
\label{adapt Fu}
Suppose that $F$ is non-archimedean.
We shall transform  the stable integration
over $U_n$ as a subgroup of  
$\widetilde{\mathrm{Sp}}_n\left(F\right)$ in \eqref{e: weil hermitian17}  into an integration
over a subgroup of $R_\lambda$ by
adapting the global argument in 
\cite[p.97--98]{Fu}
to our local setting
and shall reduce Proposition~\ref{l: main lemma} to Lemma~\ref{lem; final step} below.

Recall the $U_n$ is the group of upper unipotent matrices in $\mathrm{GL}_n\left(F\right)$.
Let us identify $U_{n-1}$ with the subgroup $\left\{\begin{pmatrix}u&0\\0&1 \end{pmatrix} : u \in U_{n-1} \right\}$
of $U_n$. Let $U_0$ be the subgroup of $U_n$ defined by 
\[
U_0=\left\{
\begin{pmatrix}1_{n-1}&a\\0&1\end{pmatrix}\in U_n
\right\}.
\]
Thus we have $U_n=U_0\rtimes U_{n-1}$.
We note that for 
\begin{equation}\label{e: u prime}
u^\prime=\begin{pmatrix}1_{n-1}&a\\0&1\end{pmatrix}\in
U_0\quad\text{with $a=\begin{pmatrix}a_1\\ \vdots \\ a_{n-1}\end{pmatrix}$}
\end{equation}
and $u_1\in U_{n-1}$, we have 
\[
\omega_\psi(g, m(u^\prime u_1)) \varphi(x_0)
=\omega_\psi(g, m(u_1)) \varphi \left(e_{-1}, \dots, e_{-n+1}, e_\lambda+\sum_{j=1}^{n-1}a_j e_{-j} \right)
\]
by \eqref{weil action-b}. 
For $u^\prime\in U_0$ of the form \eqref{e: u prime}, let
\[
s\left(u^\prime\right):=s_n\left(a_{n-1}\right)\cdots s_1\left(a_1\right).
\]
We recall that $s_j(a)$ is defined by \eqref{e: s_j}.
Then by \eqref{weil action-a}, we have
\[
\omega_\psi(g, m(u_1)) \varphi \left(e_{-1}, \dots, e_{-n+1}, e_\lambda+\sum_{j=1}^{n-1}a_j e_{-j} \right)
= \omega_\psi(s\left(u^\prime\right)^{-1}g, m(u_1))  \varphi
(x_0).
\]
Further we note that by \eqref{weil action-a} and \eqref{weil action-b},
we have
\[
\omega_\psi(s\left(u^\prime\right)^{-1}g, m(u_1))  \varphi
(x_0)=
\omega_\psi\left(\check{u}_1^{-1}s\left(u^\prime\right)^{-1}g,
1\right)\varphi\left(x_0\right)
\]
where $\check{u}_1$ is defined by  \eqref{e: u check} for $u_1\in U_{n-1}$.
We also note that 
\[
\psi_\lambda\left(m\left(u^\prime u_1\right)\right)=
\chi_\lambda\left(s\left(u^\prime\right) \check{u}_1\right)
\]
by \eqref{e: character identity}.
Hence the integral of the right hand side of \eqref{e: weil hermitian17} is equal to
\begin{multline}
\label{eq0}
\int_{U_n}^{\mathrm{st}} \int_{R_\lambda^\prime
\backslash G}\int_{R_\lambda^\prime\backslash G}\int_{R_\lambda^\prime}
\omega_\psi\left(\check{u}_1^{-1}s\left(u^\prime\right)^{-1}g,
1\right)\varphi\left(x_0\right)
\, \overline{\varphi^\prime\left(h^{-1} \cdot x_0\right)}
\\
\times \langle\pi\left(r^\prime g\right)\phi,\pi\left(h\right)\phi^\prime\rangle
\,\chi_\lambda\left(s\left(u^\prime\right) \check{u}_1\right)^{-1}\, dr^\prime\,
dg\, dh \,du
\end{multline}
where $u=u^\prime u_1$, $u^\prime\in U_0$ and 
$u_1\in U_{n-1}$.
We note the following elementary lemma.
%
\begin{lemma}\label{l: compact}
\begin{enumerate}
\item
For a given compact open subgroup $U_{n-1}^\circ$ of $U_{n-1}$ and  a compact 
open subgroup $U_0^\prime$ of $U_0$, there exists a compact open 
subgroup $U_0^\circ$ of $U_0$ such that
$U_{n-1}^\circ U_0^\circ$ is a subgroup of $U_n$
and $U_0^\circ\supset U_0^\prime$.
\item
For a given compact open subgroup $U^\circ$ of $U_n$,
there exist a compact open subgroup $U_{n-1}^\circ$ of $U_{n-1}$
and a compact open subgroup $U_0^\circ$ of $U_0$
such that $U_{n-1}^\circ U_0^\circ$ is a subgroup of $U_n$ containing 
$U_n^\circ$.
\end{enumerate}
\end{lemma}
%
\begin{proof}
For a positive integer $r$, let $U_0^{(r)}=\left\{\begin{pmatrix}1_{n-1}&a\\0&1\end{pmatrix}
: a\in \varpi^{-r}\begin{pmatrix}\mathcal O\\ \vdots \\ \mathcal O\end{pmatrix}\right\}$.
\begin{enumerate}
\item
Take $r$ sufficiently large so that 
$U_0^{(r)}\supset U_0^\prime$.
Since $U^\circ_{n-1}$ is compact, there exists an integer $s$ such that
all entries of elements of $U_{n-1}^\circ$ are in $\varpi^{-s}\mathcal O$.
Let us take integers $r_1, \cdots , r_{n-1}$ inductively so that
$r_{n-1}=r$ and $r_{n-k}\ge \mathrm{max}\left\{r, sr_{n-1},\cdots , s r_{n-k+1}\right\}$
for $2\le k\le n-1$.
Let 
\begin{equation}\label{e: compact 0}
U_0^\circ=
\left\{\begin{pmatrix}1_{n-1}&a\\0&1\end{pmatrix}:
a\in\begin{pmatrix}\varpi^{-r_1}\mathcal O\\ \vdots \\
\varpi^{-r_{n-1}}\mathcal O\end{pmatrix}\right\}.
\end{equation}
Then $U_{n-1}^\circ U_0^\circ$ is a subgroup of $U_{n}$ and $U_0^\circ\supset
U_0^\prime$.
\item
Let $U_{n-1}^\circ=U^\circ\cap U_{n-1}$.
Then $U_{n-1}^\circ$ is a compact open subgroup of $U_{n-1}$.
Since  $\displaystyle{U^\circ\subset\, \cup_{r\ge 1} U_{n-1}^\circ U_0^{(r)}}$ and 
$U^\circ$ is compact, we have $U^\circ\subset U_{n-1}^\circ U_0^{(r)}$
for $r$ sufficiently large.
By (1), we may take a compact open subgroup of $U_0^\circ$ of $U_0$
so that $U_{n-1}^\circ U_0^\circ$ is a subgroup $U_{n}$
and $U_0^\circ\supset U_0^{(r)}$.
\end{enumerate}
\end{proof}
%
By the definition of the stable integration,
for any sufficiently large compact open subgroup $U^\circ$ of $U_n$, 
the integral \eqref{eq0}  is equal to
\begin{multline}
\label{eq1}
\int_{U_n} \int_{R_\lambda^\prime
\backslash G}\int_{R_\lambda^\prime\backslash G}\int_{R_\lambda^\prime}
\chi_{U^\circ}\left(u^\prime u_1\right)\cdot
\omega_\psi\left(\check{u}_1^{-1}s\left(u^\prime\right)^{-1}g,
1\right)\varphi\left(x_0\right)
\, \overline{\varphi^\prime\left(h^{-1} \cdot x_0\right)}
\\
\times \langle\pi\left(r^\prime g\right)\phi,\pi\left(h\right)\phi^\prime\rangle
\,\chi_\lambda\left(s\left(u^\prime\right) \check{u}_1\right)^{-1}\, dr^\prime\,
dg\, dh \,du^\prime \, du_1
\end{multline}
where $\chi_{U^\circ}$ is the characteristic function of $U^\circ$.
By Lemma~\ref{l: compact}, we may take a compact open subgroup $U_{n-1}^\circ$ of $U_{n-1}$
and  a compact open subgroup $U_0^\circ$ of $U_0$ of the form
\eqref{e: compact 0} so that $U_{n-1}^\circ U_0^\circ$ 
is a compact open subgroup of $U_{n}$ and 
$U_{n-1}^\circ U_0^\circ\supset U^\circ$.
Then \eqref{eq1} is equal to
\begin{multline}\label{eq1-a}
\int_{U_{n-1}^\circ} \int_{U_0^\circ} \int_{R_\lambda^\prime
\backslash G}\int_{R_\lambda^\prime\backslash G}\int_{R_\lambda^\prime}
\omega_\psi\left(\check{u}_1^{-1}s\left(u^\prime\right)^{-1}g,
1\right)\varphi\left(x_0\right)
\, \overline{\varphi^\prime\left(h^{-1} \cdot x_0\right)}
\\
\times \langle\pi\left(r^\prime g\right)\phi,\pi\left(h\right)\phi^\prime\rangle
\,\chi_\lambda\left(s\left(u^\prime\right) \check{u}_1\right)^{-1}\, dr^\prime\,
dg\, dh \,du^\prime \, du_1.
\end{multline}
Since  the argument to obtain \eqref{e: weil hermitian17} ensures 
the absolute convergence of the most inner triple integral
and the outer double integral is over a compact group $U_{n-1}^\circ
U_0^\circ$, the integral \eqref{eq1-a} converges absolutely.
Hence we may change the order of integration
and we obtain
\begin{multline}\label{suff U1}
\mathcal W\left(
\theta\left(\varphi\otimes\phi
\right),
\theta\left(\varphi^\prime\otimes\phi^\prime\right)
\right)
=C_{E \slash F} \cdot \int_{R_\lambda^\prime
\backslash G}\int_{R_\lambda^\prime\backslash G} \int_{R_\lambda^\prime} \int_{U_0^\circ} \int_{U_{n-1}^\circ}
\\
 \left(\omega_\psi\left(g \right)\varphi\right)\left(x_0\right)\overline{ \left(\omega_\psi\left(h \right)\varphi^\prime\right)\left(x_0\right)}
\, \langle\pi\left(r^\prime s(u^\prime) \check{u}_1 g\right)\phi,\pi\left(h\right)\phi^\prime\rangle
\\
\chi_\lambda\left(s\left(u^\prime\right) \check{u}_1\right)^{-1} \, 
du_1\, du^\prime\, dr^\prime 
\,dg\, dh .
\end{multline}
Then  Proposition~\ref{l: main lemma} is reduced to the following lemma.
\begin{lemma}
\label{lem; final step}
Keep the above notation. Then
\begin{multline}\label{e: final step}
\int_{R_\lambda^\prime} \int_{U_0^\circ} \int_{U_{n-1}^\circ}
 \langle\pi\left(r^\prime s(u^\prime)  \check{u}_1g\right)\phi,\pi\left(h\right)\phi^\prime\rangle
\,\chi_\lambda( s(u^\prime)  \check{u}_1)^{-1}\, du_1\,du^\prime \, dr^\prime
\\
=\alpha\left(\pi(g)\phi, \pi(h)\phi^\prime \right).
\end{multline}
\end{lemma}
Assume that \eqref{e: final step} holds.
Then by replacing the the most inner triple integral of \eqref{suff U1}
by $\alpha\left(\pi(g)\phi, \pi(h)\phi^\prime \right)$,
we obtain the equality \eqref{e: weil hermitian13} in
Proposition~\ref{l: main lemma}.
%
%
%
%
%
\subsubsection{Proof of Lemma~\ref{lem; final step}}
\label{s; final step}
Let us prove Lemma~\ref{lem; final step} and complete
 the proof of Proposition~\ref{prp: local equality} in the non-archimedean case.

Since $U_0^\circ$ and $U_{n-1}^\circ$ are compact,
the integral on the left hand side of \eqref{e: final step} converges absolutely
by Remark~\ref{rem; abs conv}.
Hence by noting that $R_\lambda^\prime = D_\lambda S_\lambda^\prime$
where $S_\lambda^\prime$ is as given in \eqref{e: stabilizer1}
and by changing the order of integration, we have
%
\begin{multline}\label{e: integral 1}
\int_{R_\lambda^\prime} \int_{U_0^\circ} \int_{U_{n-1}^\circ}
 \langle\pi\left(r^\prime s(u^\prime)  \check{u}_1g\right)\phi,\pi\left(h\right)\phi^\prime\rangle
\,\chi_\lambda( s(u^\prime)  \check{u}_1)^{-1}\, du_1\,du^\prime \, dr^\prime
\\
=\int_{D_\lambda} \int_{U_{n-1}^\circ}
\int_{U_0^\circ}\int_{S_\lambda^\prime}
 \langle\pi\left(s_0 s\left(u^\prime\right) \check{u}_1t
 g\right)\phi,\pi\left(h\right)\phi^\prime\rangle
\chi_\lambda\left(s_0 s\left(u^\prime\right) \check{u}_1 \right)^{-1}ds_0
\, du^\prime\, du_1\, dt.
\end{multline}
%

Let us define an open subgroup $S^\sharp$ of $S^\prime$ by
\[
S^\sharp:=
\left\{\begin{pmatrix}1_{n-1}&A&B\\0&1_3&A^\prime\\0&0&1_{n-1}\end{pmatrix}
: Ae_\lambda\in \varpi^{-r_1}\mathcal O e_{-1}
+\cdots+\varpi^{-r_{n-1}}\mathcal O e_{-n+1}
\right\}
\]
with $r_i$ given in \eqref{e: compact 0}.
Then by considering a filtration of $S^\sharp$ given by
\[
S_\lambda^\prime\lhd \left(S_1^\prime\cap
S^\sharp\right)\lhd\cdots\lhd \left(S_{n-1}^\prime\cap S^\sharp\right)=S^\sharp
\]
induced from \eqref{e: filtration of S prime}
and by taking \eqref{e: S quotient} into account,
the integral \eqref{e: integral 1}
is equal to 
\begin{equation}\label{e: stable 1}
\int_{D_\lambda} \int_{S^\star}
 \langle\pi\left(st
 g\right)\phi,\pi\left(h\right)\phi^\prime\rangle
\,\chi_\lambda\left(s\right)^{-1}\,ds\, dt.
\end{equation}
Here $S^\star$ is an open subgroup of $S$ given by
$S^\star=\check{U}_{n-1}^\circ S^\sharp $
where $\check{U}_{n-1}^\circ$ is a subgroup
$\left\{\check{u}: u\in U_{n-1}^\circ\right\}$ of $S^{\prime\prime}$.
Hence, 
by the definition \eqref{e: local integral 1} of $\alpha \left(\pi(g)\phi,\pi(h)\phi^\prime \right)$,
it suffices for us to show
\begin{equation}\label{e: at last}
 \int_{S^\star}
 \langle\pi\left(st
 g\right)\phi,\pi\left(h\right)\phi^\prime\rangle
\,\chi_\lambda\left(s\right)^{-1}\,ds
=\int_{S}^{\mathrm{st}}
\langle\pi\left(stg\right)\phi,\pi\left(h\right)\phi^\prime\rangle
\,\chi_\lambda \left(s \right)^{-1}\, ds
\end{equation}
 in order to prove Lemma~\ref{lem; final step}.

We recall that the integrand of \eqref{e: stable 1} has a stable integral
over $S$ by \cite[Proposition~3.1]{Liu}.
Hence
 there exists a compact open subgroup $S^\flat$ of $S$ such that
for any compact open subgroups $S^\circ$ of $S$ containing 
$S^\flat$, we have 
\[
\int_{S}^{\mathrm{st}}
\langle\pi\left(stg\right)\phi,\pi\left(h\right)\phi^\prime\rangle
\,\chi_\lambda \left(s \right)^{-1}\, ds
=
\int_{S^\circ}
\langle\pi\left(stg\right)\phi,\pi\left(h\right)\phi^\prime\rangle
\,\chi_\lambda \left(s \right)^{-1}\, ds.
\]
By taking $U_{n-1}^\circ$ and $r_i$ sufficiently large,
we may suppose that $S^\flat\subset S^\star$.
Then for any compact open subgroups $S^\circ$ of $S$ containing 
$S^\flat$, we have 
\[
\int_{S^\circ\cap S^\star}
\langle\pi\left(stg\right)\phi,\pi\left(h\right)\phi^\prime\rangle
\,\chi_\lambda \left(s \right)^{-1}\, ds
\\
=\int_{S}^{\mathrm{st}}
\langle\pi\left(stg\right)\phi,\pi\left(h\right)\phi^\prime\rangle
\,\chi_\lambda \left(s \right)^{-1}\, ds
\]
 since $S^\circ\cap S^\star$ is a compact open subgroup of $S$ 
containing $S^\flat$.
Let $f$ denote a function on $S$ defined 
by
\[
f\left(s\right):=\chi_{S^\star}\left(s\right)\cdot
\langle\pi\left(stg\right)\phi,\pi\left(h\right)\phi^\prime\rangle
\,\chi_\lambda \left(s \right)^{-1}
\]
for $s\in S$, where $\chi_{S^\star}$ 
denotes the characteristic function of $S^\star$.
Then we have
\begin{align*}
\int_{S^\circ}f\left(s\right)\, ds=&\int_{S^\circ\cap S^\star}
\langle\pi\left(stg\right)\phi,\pi\left(h\right)\phi^\prime\rangle
\,\chi_\lambda \left(s \right)^{-1}\, ds
\\
=&\int_{S}^{\mathrm{st}}
\langle\pi\left(stg\right)\phi,\pi\left(h\right)\phi^\prime\rangle
\,\chi_\lambda \left(s \right)^{-1}\, ds.
\end{align*}
This implies that $f$ has a stable integral over $S$ and we have
\[
\int_{S}^{\mathrm{st}}f\left(s\right)\, ds=\int_{S}^{\mathrm{st}}
\langle\pi\left(stg\right)\phi,\pi\left(h\right)\phi^\prime\rangle
\,\chi_\lambda \left(s \right)^{-1}\, ds.
\]
Hence by applying Remark~\ref{r: stable integral} to $f$, we have
\eqref{e: at last}.
This completes the proof of Lemma~\ref{lem; final step}
and the proof of Proposition~\ref{prp: local equality} in the non-archimedean case.
%
%
\subsubsection{Archimedean case}
Suppose that $F$ is archimedean.
Since  $\mathcal{X}_\lambda$ is locally closed in $V^n$, 
the function $R_\lambda^\prime \backslash G \ni g \mapsto \varphi(g^{-1} \cdot x_0 )$
is compactly  supported for any $\varphi \in C_c^\infty(V^n)$. Therefore, by  Liu~\cite[Proposition~3.5]{Liu}, 
the integral 
\begin{equation}\label{e: archimedean}
\int_{R_\lambda^\prime \backslash G} \int_{R_\lambda^\prime \backslash G}
\int_{D_\lambda}
\int_{S}
\langle\pi\left(st g\right)\phi,\pi\left(h\right)\phi^\prime\rangle
\,\chi_\lambda \left(s \right)^{-1}\, ds \,dt \, dg \, dh
\end{equation}
converges absolutely.
Then we may change the order of integration in \eqref{e: archimedean}
and, by an argument
similar to the one in \ref{adapt Fu} and \ref{s; final step} in the non-archimedean case, 
we may show that the integral \eqref{e: archimedean} is equal to the right hand side of
\eqref{e: weil hermitian17}.
Then the equality \eqref{e: weil hermitian13} readily follows
and Proposition~\ref{l: main lemma} is proved also in the archimedean case.
%
%
%
\subsubsection{Corollary of Proposition~\ref{l: main lemma}}
We note the following,
which is a local counterpart of
\cite[Proposition~2 and 3]{FM0},
as a corollary of 
Proposition~\ref{l: main lemma}.
%
\begin{corollary}\label{c: genericity}
Let $\pi$ be an irreducible unitary representation of $G$.
Suppose that $\pi$ is tempered when $F$ is non-archimedean and 
$\pi$ is a discrete series representation when $F$
is archimedean.
Then for 
$\sigma=\theta\left(\pi,\psi\right)$, we have
\begin{equation}\label{e: theta genericity}
\mathrm{Hom}_N\left(\sigma,\psi_\lambda\right)\ne
\left\{0\right\}
\Longleftrightarrow
\alpha\not\equiv 0.
\end{equation}
\end{corollary}
%
\begin{proof}
By Lapid and Mao~\cite[Proposition~2.10]{LM0},
$\mathrm{Hom}_N\left(\sigma,\psi_\lambda\right)\ne
\left\{0\right\}$ implies that $\mathcal W$ defined by \eqref{e: weil hermitian4}
is not identically zero.
Then \eqref{e: weil hermitian13} clearly implies that
$\alpha$ is not identically zero.
Conversely suppose that there exist $\phi,\phi^\prime\in V_\pi$
such that $\alpha\left(\phi,\phi^\prime\right)\ne 0$.
Then by 
Lemma~\ref{l: non-vanishing} and Proposition~\ref{p: local equality},
$\mathcal W$ is not identically zero
and it clearly implies that $\sigma$ is $\psi_\lambda$-generic,
i.e. $\mathrm{Hom}_N\left(\sigma,\psi_\lambda\right)\ne
\left\{0\right\}$, by \eqref{e: weil hermitian5}.
\end{proof}
%
%
%
%
%
%
%
%
%
%
\subsection{Proof of the statement (1) of Theorem~\ref{t: main theorem}}
\label{ss: local GGP real}
%
We return to the global setting.
As we noted in
\eqref{e: theta lift},  $\sigma=\Theta_n\left(\pi,\psi\right)$ is $\psi_\lambda$-generic when $B_{\lambda,\psi}\not\equiv 0$.
Hence its local component $\sigma_v$ is $\psi_{\lambda,v}$-generic at
every place $v$ of $F$.
Thus at any place $v$ of $F$, $\alpha_v$ does not vanish identically 
by Corollary~\ref{c: genericity}.
%
%
%
%
%
%
%
%
%
%
\section{Proof of Corollary~\ref{maincor}}\label{s: maincor}
%
%
By Theorem~\ref{t: main theorem}, it is enough for us to
show that
the right hand side of \eqref{e: main identity}
vanishes identically when $B_{\lambda,\psi}\equiv 0$.
Suppose on the contrary.
Then in particular  $L\left(1/2,\pi\right)L\left(1/2,\pi\times\chi_E\right)\ne 0$.
By the assumption that Conjecture~9.5.4 in Arthur~\cite{Ar} holds for
any group in $\mathcal G$, 
$\pi$ has a weak lift to $\mathrm{GL}_{2n}(\mA)$.
Then the global descent method by Ginzburg, Rallis and Soudry~\cite{GRS11} 
gives an irreducible cuspidal globally generic automorphic representation
$\pi^\circ$ of $\mathbb G\left(\mA\right)$ which is nearly equivalent to $\pi$.
Thus Proposition~5 in \cite{FM0} is applicable to $\pi$.
Hence there exist 
$G^\prime=\mathrm{SO}\left(V^\prime\right)\in\mathcal G$
where $\mathrm{disc}\left(V^\prime\right)=\left(-1\right)^n$
and an irreducible cuspidal automorphic representation 
$\pi^\prime$ of $G^\prime\left(\mA\right)$ 
having the special Bessel model of type $E$,
which is nearly equivalent
to $\pi$.
We shall reach a contradiction by
showing that $G=G^\prime$ and $\pi=\pi^\prime$.

Since $B_{\lambda,\psi}\not\equiv 0$ on
$V_{\pi^\prime}$, 
$\Theta_n\left(\pi^\prime,\psi\right)$,
the theta lift of $\pi^\prime$ to $\widetilde{\mathrm{Sp}}_{n}(\mA)$ with respect to $\psi$
is $\psi_\lambda$-generic by \cite[Proposition~2]{FM0}.
In particular $\theta\left(\pi^\prime_v,\psi_v\right)$
is $\psi_{\lambda, v}$-generic for any $v$.
On the other hand, we have $\alpha_v\not\equiv 0$  on $V_{\pi_v}$
since the right hand side of \eqref{e: main identity} is not identically zero.
Hence $\theta\left(\pi_v,\psi_v\right)$ is also $\psi_{\lambda,v}$-generic
for any $v$ by 
Corollary~\ref{c: genericity}.

Suppose that $v$ is finite.
Since $\pi$ and $\pi^\prime$ are nearly equivalent,
it is readily shown that they have the same $A$-parameter 
by an argument  similar to the one  in Atobe and 
Gan~\cite{AG}.
Further the temperedness of $\pi$ implies that
$\pi$ and $\pi^\prime$ share the same local $L$-parameter
at each finite place.
Here we recall the assumption that the local Langlands correspondence \cite[Conjecture~9.4.2]{Ar} holds for any element of $\mathcal{G}$.
Since $\pi_v$ and $\pi^\prime_v$ both have the special
Bessel model of type $E_v$, 
we have $G_v\simeq G^\prime_v$ and
$\pi_v\simeq\pi^\prime_v$ by Waldspurger~\cite{Wa1, Wa2}.

When $v$ is real, $\theta\left(\pi_v,\psi_v\right)$
and $\theta\left(\pi^\prime_v,\psi_v\right)$ have the same $L$-parameter
by Adams and Barbasch~\cite{AB}.
Then we have $\theta\left(\pi_v,\psi_v\right)\simeq
\theta\left(\pi^\prime_v,\psi_v\right)$ by the uniqueness
of generic element in tempered $L$-packets
(see Kostant~\cite{Ko}, Shelstad~\cite{Sh}
and Vogan~\cite{Vo}).
Since $V$ and $V^\prime$ have the same discriminant,
we have $G_v\simeq G^\prime_v$ by \cite{AB}.
Hence by the Howe duality, we have $\pi_v\simeq\pi^\prime_v$.

Thus we have shown that $G_v\simeq G^\prime_v$ and $\pi_v\simeq \pi^\prime_v$
for any place $v$ of $F$.
Hence we have $G=G^\prime$ and $\pi\simeq\pi^\prime$.
The latter actually implies that $\pi=\pi^\prime$ since the multiplicity of $\pi$
is one by Arthur~\cite[Conjecture~9.5.4]{Ar}.
%
%
%
%
%
%
%
%
%

%
%
%
%
%
%
%
%
%
%
%
%
%
%
\end{document}